\documentclass[11pt,a4paper]{amsart}

\usepackage{graphicx,url,etoolbox}

\makeatletter
\patchcmd{\@settitle}{center}{flushleft}{}{}
\patchcmd{\@settitle}{center}{flushleft}{}{}
\patchcmd{\@setauthors}{\centering}{\raggedright}{}{}
\patchcmd{\abstract}{3pc}{0pt}{}{} 
\makeatother
\usepackage{mathtools}
\usepackage{amsmath,amssymb,amsfonts,color}
\usepackage{algorithm}
\usepackage{algorithmic}
\usepackage{color}
\usepackage{booktabs}
\usepackage[mathscr]{eucal}
\usepackage{url}
\usepackage{dsfont}
\usepackage{subfloat}
\usepackage{subcaption}
\usepackage{tabularx}
\usepackage{hyperref}
\usepackage{caption}
\usepackage{subcaption,graphicx}
\usepackage{bm}
\usepackage{epstopdf}
\usepackage{xargs}
\usepackage[dvipsnames]{xcolor}
\calclayout
\numberwithin{equation}{section}

\newtheorem{theorem}{Theorem}[section]
\newtheorem{lemma}{Lemma}

\newtheorem{remark}{Remark}
\newtheorem{definition}{Definition}
\newtheorem{assumption}{Assumption}

\makeatletter
\newcommand\makebig[2]{%
	\@xp\newcommand\@xp*\csname#1\endcsname{\bBigg@{#2}}%
	\@xp\newcommand\@xp*\csname#1l\endcsname{\@xp\mathopen\csname#1\endcsname}%
	\@xp\newcommand\@xp*\csname#1r\endcsname{\@xp\mathclose\csname#1\endcsname}%
}
\makeatother
\makebig{biggg} {3.0}
\makebig{Biggg} {4}
\makebig{bigggg}{5}
\makebig{Bigggg}{6}

\newcommand{\cF}{\mathcal{F}}
\newcommand{\R}{\mathbb{R}}

\newcommand{\cP}{\mathcal{P}}

\newcommand{\cU}{\mathcal{U}}

\newcommand{\cJ}{\mathcal{J}}
\newcommand{\cH}{\mathcal{H}}
\newcommand{\cL}{\mathcal{L}}
\newcommand{\cW}{\mathcal{W}}
\newcommand{\E}{\mathbb{E}}
\newcommand{\px}[1]{\partial_{{x} _{#1}}}

\newcommand{\norm}[1]{\left\lVert #1\right\rVert}

\def\one{\mbox{1\hspace{-4.25pt}\fontsize{12}{14.4}\selectfont\textrm{1}}}

\definecolor{philipp}{RGB}{205, 102, 000}

\usepackage[colorinlistoftodos,prependcaption,textsize=tiny]{todonotes}

\begin{document}
	\title[Policy iteration for HJI equations with control constraints]{Policy iteration for Hamilton-Jacobi-Isaacs equations with control constraints and comparison with Hamilton-Jacobi-Bellman equations}

	\author{Sudeep Kundu\textsuperscript{$\dagger$}}
	\thanks{\textsuperscript{$\dagger$}Department of Mathematical Sciences, Rajiv Gandhi Institute of Petroleum Technology, Jais-229304, Uttar Pradesh, India, ({\tt
			sudeep.kundu@rgipt.ac.in}).}
	\author {Karl Kunisch\textsuperscript{$*$}}
	\thanks{\textsuperscript{$*$}Institute for Mathematics and Scientific Computing, University of Graz, Heinrichstrasse 36, A-8010 Graz, Austria and
		Radon Institute for Computational and Applied Mathematics (RICAM), Altenbergerstra{\ss}e 69, A-4040 Linz, Austria, ({\tt
			karl.kunisch@uni-graz.at}).}
	\maketitle
	\begin{abstract}
		Convergence of the bilevel policy iteration algorithm for the value function, satisfying the Hamilton-Jacobi-Isaacs (HJI) equation, is analyzed in the presence of control constraints. Both first and second-order HJI equations corresponding to the deterministic and stochastic systems  are considered. For numerical tests,  a semi-implicit upwind scheme for backward HJI PDEs is applied and a thorough comparison between the solutions to first and second-order Hamilton-Jacobi-Bellman (HJB) equations is presented for both the unconstrained and the constrained control cases. 
		
	\end{abstract}
	
	
	%
	
	

	\section{Introduction}\label{intro}

	The Dynamic Programming Principle (DPP) is an important tool in optimal control theory
	since it opens the door for feedback or closed loop control.
	It evolves around the value function, which is the  viscosity solution of a Hamilton-Jacobi (HJ) type partial differential equation (PDE) \cite{BCD97}. In  open-loop control, the control depends only on time and the initial state. It is not given as a function of the state at time $t$. This has several disadvantages. If, for example, the initial state changes, the optimal control has to be recomputed  from scratch. Also, open loop control is not designed to  handle modeling errors or exogenous disturbances, or other variations of system parameters. In this scenario, feedback or closed-loop control has several advantages, since it explicitly depends on the state of the system. Moreover, robust optimal feedback control is essential in many scenarios of practical interest, including, for instance, robotics and flight control, due to its enhanced disturbance rejection properties. $\cH_\infty$ robust control \cite{DG89} - the name originated from the Hardy space of analytic functions bounded in the right half-plane -  is a min-max optimization technique, also known as convex-concave games, which is the focus in this paper.

	In the  case of linear systems with quadratic cost,  $\cH_2$ and $\cH_\infty$ control problems can be solved on the basis of the regulator or the game-theoretic  Riccati equations.
	For nonlinear systems or  non-quadratic cost, such complete solution strategies are not available. Linear-quadratic approximations may prove effective in a local analysis for certain cases, but that may step short in others since critical information of the nonlinearity gets ignored.  In this paper, we are interested in directly controlling  the nonlinear system. When there is only one player (control or input) the optimal   value function satisfies the Hamilton-Jacobi-Bellman (HJB) equation related to the $\cH_2$ control problem.  Two-person zero-sum noncooperative differential games, where one player or agent (control) wants to minimize the cost or payoff function and another player (disturbance) wants to maximize the cost, are closely related to the $\cH_\infty$ control which in turn satisfies the Hamilton Jacobi Isaacs (HJI) equation \cite{V92}. Regarding details on  $\cH_\infty$ control,  its relation to the Isaacs equation, and to two-player zero-sum differential game theory, we refer to \cite{S96, S99, BO98}.
	Solving HJ equations is frequently addressed by successive approximation or policy iteration, which are essentially equivalent to applying a Newton iteration. In the case of infinite horizon or stabilization problems, a  proper initialization in the form of a stabilizing control is required. Whenever an initial stabilizing control is not available, it is common to use a discount factor in the cost. Once it is solved, we obtain the control in the HJB case, and the  control as well as the  disturbance in the HJI case. For differentiability of the value function, some notable results are given in  \cite{G05} for linear systems and in \cite{CF13} for nonlinear systems.

	A variety of  numerical techniques are available in the literature to solve for HJB or HJI equations, including \cite{CL84}. Policy iteration, mentioned already above,   for the HJB equation was developed  by several authors \cite{AFK15, AFS18, BM98, BST98} for unconstrained control and in \cite{AL05, kk20} for the case with control constraints. While in  \cite{kk20} the constraints are realized as hard constraints, they  are approximated via non-quadratic penalty functionals  in \cite{AL05} and earlier in \cite{L98}. In \cite{Aal17} numerical techniques related to finite difference schemes are discussed. Other techniques include semi-Lagrangian schemes \cite{FF14}, domain decomposition methods \cite{CCFP12}, penalty approach with finite difference methods \cite{WR11}, the finite element methods \cite{GR85}, and level set methods \cite{OF03}. Policy iteration for the HJI equation has likely  first been  addressed  in \cite{BM98}. More recent results include   polynomial approximation methods \cite{kk2018, kkk20}, tensor calculus techniques \cite{BSU16, DKK19} and deep neural network techniques \cite{KW20, DLM20}, reinforcement learning techniques \cite{B19}, graph tree structures  \cite{AFS18}, sparse grids  \cite{GK16} and splitting methods \cite{LCO18}.
	
		%
		%
		%
		%
	%
	HJI equations arise in stochastic differential games, providing the value function for zero-sum settings. In systems like the stochastic Lorenz equation and  Van der Pol oscillator, they capture the minimax interaction between stabilizing control and destabilizing disturbances. This framework explains uncertainty in chaotic, noisy, and nonlinear dynamics, with applications ranging from engineering structures and secure communications to biological networks, finance, and climate systems. Another typical setting models a central agent (leader) influencing many followers, each striving to remain close while subject to random disturbances \cite{YGK25}.
	For stochastic systems, $\cH_\infty$ control was discussed in \cite{ZC06} both for the finite and infinite horizon problems. Policy iteration, also known as Howard's algorithm, in the stochastic context,   was considered in \cite{WS02} for the  unconstrained control  case  and in \cite{kk20} for the case with constraints. Neural network based policy iteration for the second-order HJI equations was presented recently in \cite{IRZ20}. A deep neural network technique to solve zero-sum stochastic differential games was applied in \cite{RZ20}. For second-order linear HJB, tensor calculus techniques were considered in \cite{YPL}. The Markov chain approximation method was used in \cite{BZ03} to solve the stochastic HJB equation  in conjunction with policy iteration and the finite difference method. Finally, we refer to the monograph \cite{KD01} on numerical methods for stochastic control problems.
	
	In this article, we address selected aspects concerning the  convergence of policy iteration related to two-player zero-sum differential games, in the presence of  control. In the unconstrained control case, this was addressed in \cite{BM98} for deterministic systems. In \cite{ALH06} control constraints are realized by a penalty technique.  Here, we employ a projection operator to  strictly enforce convex control constraints.  We show convergence of the value function for both first and second-order  HJI equations. With a semi-implicit upwind scheme, we solve the HJI equation; and with an application of the bisection algorithm, we obtain the smallest value of the disturbance attenuation level, to be  defined below. We test our results for the controlled  Van der Pol and Lorenz systems and solve them for both the $\cH_2$ and $\cH_\infty$ formulations. A  detailed comparison between the HJB and HJI synthesis in terms of the performance index, the evolution of the state, and the control trajectory is provided. To the best of our knowledge, such an analysis is not available in the literature for  HJI equations with  hard constraints.

	The rest of the paper is divided as follows. Section $2$ is dedicated to the convergence of the policy iteration algorithm for the first order HJI equation under input constraints. Section $3$ presents the corresponding results for  the stochastic case. Section 4 contains the numerical results, and Section $5$ provides some concluding remarks.
	
	\section{Nonlinear $\cH_\infty$ control for deterministic systems}
	We consider the following infinite horizon optimal robust control problem:
	\begin{align}\label{eq1.1}
		\underset{u(\cdot)\in \cU}{\min}\;\underset{w(\cdot)\in \cW}{\max}\;&\cJ(u,w;x):=\int\limits_0^\infty \Big(\ell(y(t))+\|u(t)\|_R^2-\gamma^2\|w(t)\|_P^2\Big)\, dt,
	\end{align}
	subject to the  deterministic system
	\begin{equation}\label{eq1.2}
		\dot y(t)= f(y(t))+g(y(t))u(t)+h(y(t))w(t)\,,\quad y(0)=x,
	\end{equation}
	where $y(t)=(y_1(t),\ldots,y_d(t))^t\in \R^d$ is the state vector, $w(\cdot)\in\cW$ is the disturbance signal with $\cW=\cL^2(\R_+;\R^q)$,
	$u(\cdot)\in\cU$ is the control input with $\cU=\cL^2(\R_+;U)$, where $U\subset\R^m$ is defined later.
	Further $\ell(y)>0$, $y\neq 0$ is the state running cost, with $\ell(0)=0$,  $\norm{w}^2_{P}=w^tPw$, with $P\in\R^{q\times q},\,P>0$  a positive definite matrix. The disturbance attenuation level $\gamma>0$ is known as the measure of robustness, the control penalization term $\|u\|_R^2=u^tRu$ with $R\in\R^{m\times m},\,R>0$ is a positive definite matrix. We assume that the system dynamics $f(y):\R^d\rightarrow \R^d$ with $f(0)=0$, $h(y):\R^d\rightarrow \R^{d\times q}$, and $g(y):\R^d\rightarrow \R^{d\times m}$ are uniformly Lipschitz continuous on $\R^d$, and that  $g$  and  $h$ are globally bounded.
	Since $f(0)=0$ the state $y=0$ is an equilibrium of the above system for $u=w=0$. Throughout, the initial states will be chosen from a fixed  domain $\Omega$ containing the origin in its interior.

	We assume that the control lies in a closed convex set $U$ containing the origin in $\R^m$.
	As a particular example, $U$ can be chosen as
	\begin{align}\label{eqc}
		U=\{u=(u_1,\ldots,u_m)\hspace{0.1cm} | \hspace{0.1cm}\alpha_i\leq u_i\leq \beta_i, \quad i=1,\ldots, m\},
	\end{align}
	where $\alpha=(\alpha_1,\ldots,\alpha_m)\in \R^m$, $\beta=(\beta_1,\ldots,\beta_m)\in\R^m$ and $\alpha_i\leq0,$ $\beta_i\geq 0$.
	
	In the time domain, the goal of $\cH_{\infty}$ control consists in defining a feedback control law $u(t)=u(y(t))$ such that the system \eqref{eq1.2} is closed loop asymptotically stable and has $\cL_{2}$ gain less than or equal to $\gamma$, i.e.
	\begin{equation}\label{eq1.3}
		\int\limits_{0}^{T}\Big(\ell(y(t))+\|u(t)\|^2_R\Big)\, dt\leq\gamma^2\int\limits_{0}^{T}\norm{w(t)}^2_{P}\,dt,
	\end{equation}
	for all $w\in \cL^2(0,T;\R^q)$, and all $T\geq 0$, where $y$ the solution of \eqref{eq1.2} with $x=0$.
	The $\cH_{\infty}$ norm $\gamma^*$ is defined as  the smallest value of $\gamma$  such that for $\gamma\geq \gamma^*$, \eqref{eq1.3} holds under system \eqref{eq1.2}.
	The optimal value function associated to \eqref{eq1.1},  also known as available storage function, is defined by
	\[V_\gamma(x)=\underset{u(\cdot)\in \cU}{\min}\;\underset{w(\cdot)\in \cW}{\max}\;\cJ(u,w;x).\]
	Throughout we assume that:
	\begin{assumption}\label{ass1}
		The radially unbounded value function satisfies  $V_\gamma\in C^1(\R^d)$. Further, for the policy iteration Algorithm \ref{alg:sg1}, we assume that $V^{(i,j)}_{\gamma}(x)\in C^1(\R^d)$ for each $i,\hspace{0.1cm} j\geq 0$.
	\end{assumption}
	\begin{remark}\label{r1}
		{\em The regularity assumption for $V_\gamma$ , $V^{(i,j)}_{\gamma}$ is common in the literature for related analysis. Similar assumptions were made  earlier in \cite{SL79, V92, BM98, AL05, ALH06, GK16}. The investigation of the regularity of the value function is certainly an important research topic in its own right, see also \cite{BKP19} and \cite{kp24}. In passing we note that the upper and lower solutions to the HJI equation \eqref{hji3} are known to be Lipschitz continuous and hence by Rademacher's theorem they are differentiable almost everywhere \cite[Corollary 3.8]{cp10}.}
	\end{remark}
	Due to dynamic programming principle, under the Assumption \ref{ass1},  the optimal value function,
	$V_\gamma(x)$ satisfies the HJI equation \cite[Chapter 8]{BO98},
	\begin{equation}\label{hji}
		\underset{u\in U}{min}\,\,\underset{w\in W}{max}\{ \nabla V_\gamma(x)^t(f(x)+g(x) u+h(x)w)+ \ell(x)+\|u\|_R^2-\gamma^2\|w\|_P^2\}=0\,,
	\end{equation}
	where $\nabla V_\gamma(x)=(\px{1}V_\gamma,\ldots,\px{d}V_\gamma)^t$.
	Note that since the cost \eqref{eq1.1} and the system \eqref{eq1.2} are separable with respect to control and disturbance strategies, the so called Isaacs' condition holds \cite[Chapter VIII]{BCD97} which guarantees the existence of the value function satisfying \eqref{hji}.
	
	For $\gamma\geq \gamma^*$, \eqref{hji} has a positive definite solution \cite[Appendix]{BCD97}, \cite{S99} (here positive definite means $V_\gamma(x)>0$ for $x\neq 0$ and $V_\gamma(0)=0$).
	If the initial state  is not an equilibrium solution, then instead of \eqref{eq1.3} we get
	\begin{equation}\label{eq1.4}
		\int\limits_{0}^{T}\Big(\ell(y(t))+\|u(t)\|^2_R\Big)\, dt\leq\gamma^2\int\limits_{0}^{T}\norm{w(t)}^2_{P}\;dt + V_\gamma(x),
	\end{equation}
	see e.g. \cite[Appendix]{BCD97}, \cite{V92}.

	Under Assumption \ref{ass1}, according to \eqref{hji}, the optimal control u is given  in feedback strategic form by
	$$u^*_\gamma(x)=\cP_{U}\Big(-\frac{1}{2}R^{-1}g(x)^t\nabla V_\gamma(x)\Big),$$ where $\cP_{U}$ is the orthogonal projection in the $R$-weighted inner product on $\R^m$ onto  $U$, and $V_{\gamma}$ is the corresponding HJI value function. For the particular case \eqref{eqc}, the projection $\cP_{U}$ is given by $u^*_\gamma(x)=\min\Big\{\beta,\max\{\alpha,-\frac{1}{2}R^{-1}g(x)^t\nabla V_{\gamma}(x)\}\Big\}$, where the min - and max - operations are considered coordinate wise. An  optimal disturbance $w_\gamma\in \R^q$ is given in feedback strategic form by $w^*_\gamma=\frac{1}{2\gamma^2}P^{-1}h^t\nabla V_\gamma(x)$.
	
	When there is no confusion we denote $(u_\gamma,w_\gamma) = (u^*_\gamma,w^*_\gamma)$. Moreover, in many places, we suppress the dependence of $f\;,g$ and $h$ on the state variable.
	
	Using the values of $u_\gamma$ and $w_\gamma$, the HJI equation \eqref{hji} becomes
	\begin{align}\label{hji2}
		\nabla V_\gamma(x)^t\Big(f(x)&+g(x)\cP_{U}\Big(-\frac{1}{2}R^{-1}g(x)^t\nabla V_{\gamma}(x)\Big)\Big)+ \frac{1}{4\gamma^2} \nabla V_\gamma(x)^th(x)P^{-1}h(x)^t\nabla V_\gamma(x)\notag\\
		&+\ell(x)+\norm{\cP_{U}\Big(-\frac{1}{2}R^{-1}g(x)^t\nabla V_{\gamma}(x)\Big)}^2_{R}=0 \,.
	\end{align}
	For an unconstrained control $u_\gamma\in \R^m$, the corresponding HJI equation reduces to
	\begin{align}\label{hji3}
		\nabla V_\gamma(x)^tf(x)+ \frac14 \nabla V_\gamma(x)^t\left(\frac{1}{\gamma^2}h(x)P^{-1}h(x)^t-g(x)R^{-1}g(x)^t\right)\nabla V_\gamma(x)+\ell(x)=0 \,.
	\end{align}
	Sufficient conditions for the local existence of a solution $V_\gamma\geq 0$ of \eqref{hji3} are given in \cite{V91}.
	When $\gamma\to\infty$ in the disturbance law, the HJI equations \eqref{hji2} and \eqref{hji3} result in the constrained and unconstrained HJB equations related to $\cH_2$ control. Related details are given in \cite{kk20}. The case with constraints is treated in the Appendix of \cite{BCD97}, for example. In \cite[Appendix B, Theorem 4.5] {BCD97} regularity conditions on the problem data are provided under which the existence of (viscosity) solutions to \eqref{hji2} implies the solvability of \eqref{eq1.1} and vice-versa.  Here (strict)solvability refers to the property that the unperturbed system \eqref{eq1.2} is (asymptotically) stable and the $L^2$-gain property \eqref{eq1.4} holds \cite[Appendix B.1, Theorem 4.5]{BCD97}.


	{Throughout the paper, it is assumed that $\bar \Omega$ is contained in the region of stability  related to \eqref{eq1.2} with $w=0$,  i.e. the region of initial conditions  for which trajectories are asymptotically stabilizable to the origin. In certain cases, additional condition on initial data in $\Omega$ will be specified.}

	Solving \eqref{hji3} is a challenging task, and policy iteration is one possible approach. Alternative methods include polynomial approximations to the HJI equation \cite{kkk20} where the case without constraints has been discussed.
	The policy iteration or successive approximation procedure  for solving the  HJI equation \eqref{hji2} is presented in Algorithm \ref{alg:sg1} for a fixed value $\gamma$. Finding the smallest value of $\gamma$ i.e. $\gamma^*$ is an independent task which is typically achieved by a bisection algorithm in numerical practice, see e.g. \cite{BBK89}.
	By applying the bilevel policy iteration Algorithm \ref{alg:sg1}, which can also be considered as a Newton iteration,  equation \eqref{hji2} is reduced to solving  a sequence of linear HJI equations \eqref{eq:aux1}, referred to  as Generalized HJI (GHJI) equations.
	
	\begin{algorithm}[ht!]
\begin{algorithmic}
	\STATE{{\bf Input}: Let $u_{\gamma}^{(0)}(x)$ be an asymptotically stabilizing control law  for  the dynamics \eqref{eq1.2} with $w_\gamma=0$ and $\gamma\geq\gamma^*$.}
	\STATE{\bf  For $i=0$ to $\infty$}
	\STATE{\bf \qquad Set $w^{(i,0)}_{\gamma}(x)\equiv 0$,}
	\STATE{\bf \qquad For $j=0$ to $\infty$}
	\STATE{\bf 	\hspace{1cm} Solve for $V^{(i,j)}_{\gamma}(x)\in C^1(\R^d)$ :
		\begin{equation}\label{eq:aux1}
			\Bigg\{	\begin{array}{r@{}l}
				\nabla V^{(i,j)}_{\gamma}(x)^t&\big(f(x)+g(x)u^{(i)}_{\gamma}+h(x)w^{(i,j)}_\gamma\big)+ \ell(x)+\|u^{(i)}_{\gamma}\|_R^2-\gamma^2\|w^{(i,j)}_{\gamma}\|^2_P=0\,,\\
				V^{(i,j)}_{\gamma}(0)&=0.
			\end{array}
	\end{equation}}
	\vspace{-3mm}		
	\STATE{\bf \hspace{1cm} Update the disturbance:
		\begin{equation}\label{eq:aux2}
			w^{(i,j+1)}_{\gamma}(x)=\frac{1}{2\gamma^2}P^{-1}h(x)^t\nabla V^{(i,j)}_{\gamma}(x)\,,
		\end{equation}
		\qquad	j=j+1	}
	
	\bf {\qquad End j loop
		
		\STATE{Update the control:
			\vspace{-3mm}
			\begin{equation}\label{eq:aux3}
				u^{(i+1)}_{\gamma}(x)=\cP_{U}\Big(-\frac{1}{2}R^{-1}g(x)^t\nabla V_\gamma^{(i,\infty)}(x)\,\Big).
			\end{equation}
			i=i+1	\;}
		
		End i loop}
	\caption{Continuous Policy Iteration for HJI equations}\label{alg:sg1}
\end{algorithmic}
\end{algorithm}

\begin{remark}\label{rm1}
{\em In Algorithm \ref{alg:sg1}  we tacitly assumed that the  inner loop converges. Sufficient conditions for this to occur will be given in Theorem \ref{hlm2.1}  and  Remark \ref{rm4} below.}
\end{remark}
\begin{definition}[Asymptotic Stability in the sense of Lyapunov]
The equilibrium solution $y=0$ of \eqref{eq1.2} is locally asymptotically stable if:
\begin{enumerate}
	\item It is stable in the sense of Lyapunov i.e., for every $\epsilon > 0$, there exists a $\delta(\epsilon) > 0$ such that, if $\|x\| < \delta$, then $\|y(t)\| < \epsilon$ for all $t > 0$.
	\item There exists a $\delta'  > 0$ such that, if $\|x\| < \delta'$, then $y(t) \to 0$ as $t \to \infty$.
\end{enumerate}
The equilibrium state $y=0$ of \eqref{eq1.2} is globally asymptotically stable if it is asymptotically stable for any initial condition, meaning that the condition for asymptotic stability holds for any $\delta' > 0$.
\end{definition}
\begin{remark}\label{rm2}
{\em In Algorithm \ref{alg:sg1}  it is further  assumed that  $\gamma\geq\gamma^*$.  There are no effective analytical  methods to characterize the value $\gamma^*$, \cite[page 509]{BCD97}. In numerical practice, $\gamma$ is chosen sufficiently large at first and then reduced by a bisection algorithm, as in \cite{BBK89}.}
\end{remark}

Next we establish a monotonicity result for the inner loop.  For this purpose we define for $x\in\Omega$ and given a stabilizing control $u_\gamma^{(i)}$
\begin{equation}\label{eqkk1}
\left\{
\begin{array}{ll}
	V^{(i)}_\gamma(x):=\max_{w(\cdot)\in \cW} \cJ(u_\gamma^{(i)},w;x)\\[1.7ex]
	\dot y(t)= f(y(t))+g(y(t))u^{(i)}_{\gamma}(y(t))+h(y(t))w(t)\,,\quad y(0)=x.
\end{array}
\right.
\end{equation}

\begin{theorem}\label{hlm2.1}
If the system
\begin{equation}\label{eq1}
	\dot y(t)= f(y(t))+g(y(t))u^{(i)}_{\gamma}(t)+h(y(t))w^{(i,j)}_{\gamma}(t)\,,\quad y(0)=x,
\end{equation}
is asymptotically stable in the sense of Lyapunov on $\Omega$ for all pairs $(i,j)$, where $i\geq 0$ is fixed, and all   $j\geq 0$, and $V^{(i,j)}_\gamma\in C^1(\R^d)$,
then we obtain
$0\le V_\gamma^{(i,j)}(x)\leq V_\gamma^{(i,j+1)}(x) \leq V^{(i)}_\gamma(x),$ $\forall x\in\Omega$.
Moreover $V_\gamma^{(i,j)}(x)>0$ for $x\neq 0$.
\end{theorem}
\begin{proof}
Since $\dot y=f(y)+g(y)u^{(i)}_{\gamma}+h(y)w^{(i,j)}_{\gamma}$ is asymptotically stable for each pair $(i,j)$ and utilizing that $V^{(i,j)}_\gamma(0)=0$,
the difference between  $V^{(i,j)}_\gamma(x)$ and $V^{(i,j+1)}_\gamma(x)$  can be obtained by calculating the time derivative of $V^{(i,j)}_{\gamma}$, $V^{(i,j+1)}_{\gamma}$
along the trajectory of
\begin{align}\label{eq1.6}
	\dot y&= \Big(f(y)+g(y)u^{(i)}_\gamma+h(y)w^{(i,j+1)}_\gamma\Big), \; y(0)=x \in \Omega,
\end{align}
to obtain
\begin{align*}
	&V_\gamma^{(i,j)}(x)-V_\gamma^{(i,j+1)}(x)\\
	&=\int_{0}^{\infty}\Bigg(\Big({\nabla V_\gamma^{(i,j+1)}(y)}^t\big(f(y)+g(y)u^{(i)}_\gamma+h(y)w^{(i,j+1)}_\gamma\big)\Big)\\
	&\qquad
	-\Big({\nabla V^{(i,j)}_\gamma(y)}^t\big(f(y)+g(y)u^{(i)}_\gamma+h(y)w^{(i,j+1)}_\gamma\big)\Big)\Bigg)\; dt,
\end{align*}
where in $y$, $u^{(i)}_\gamma= u^{(i)}_\gamma(y)$, and $w^{(i,j+1)}_\gamma=w^{(i,j+1)}_\gamma(y)$ the dependence on $t$ is suppressed.
From the GHJI equation \eqref{eq:aux1}, it follows that
\begin{align}\label{e2.2}
	\nabla {V^{(i,j+1)}_{\gamma}(y)}^t\Big(f(y)+g(y)u^{(i)}_{\gamma}+h(y)w^{(i,j+1)}_{\gamma}\Big)
	=-\Big(\ell(y)+\norm{u^{(i)}_{\gamma}}^2_R-\gamma^2\norm{w^{(i,j+1)}_{\gamma}}^2_P\Big),
\end{align}
and, suppressing the dependence of $f$, $g$ and $h$ on $y$,
\begin{align}\label{e2.3}
	{\nabla V^{(i,j)}_{\gamma}}^t&\Big(f+gu^{(i)}_{\gamma}+hw^{(i,j+1)}_{\gamma}\Big)\notag\\
	&={\nabla V^{(i,j)}_{\gamma}}^t\Big(f+gu^{(i)}_{\gamma}+hw^{(i,j)}_{\gamma}\Big)+\langle h^t\nabla V^{(i,j)}_{\gamma},w^{(i,j+1)}_{\gamma}-w^{(i,j)}_{\gamma}\rangle\notag\\
	&= -\Big(\ell(y)+\norm{u^{(i)}_{\gamma}}^2_R-\gamma^2\norm{w^{(i,j)}_{\gamma}}^2_P\Big)+\langle 2\gamma^2Pw^{(i,j+1)}_{\gamma},w^{(i,j+1)}_{\gamma}-w^{(i,j)}_{\gamma}\rangle,
\end{align}
where $\nabla V^{(i,j)}_{\gamma}$ is calculated along $y$.
Subtracting \eqref{e2.3} from \eqref{e2.2} we obtain
\begin{align*}
	V^{(i,j)}_{\gamma}(x)-V^{(i,j+1)}_{\gamma}(x)=-\int_{0}^{\infty}\gamma^2\norm{w^{(i,j)}_\gamma-w^{(i,j+1)}_\gamma}^2_P \;dt.
\end{align*}
Hence we get $V^{(i,j)}_{\gamma}(x)\leq V^{(i,j+1)}_{\gamma}(x)$, and the desired monotonicity is established.

To obtain the upper bound, let $V^{(i)}_\gamma$ be as defined in \eqref{eqkk1}. Calling on \eqref{eq:aux1} we obtain
\begin{equation*}
	V_\gamma^{(i,j)}(x)=\int^\infty_0 \big(\ell(y(t))+ \|u^{(i)}_\gamma(y(t))\|^2_R -\gamma^2\|w_\gamma^{(i,j)}(y(t))\|^2_P\big)\, dt = \cJ(u^{(i)}_\gamma,w^{(i,j)}_\gamma;x),
\end{equation*}
where $y=y(u^{(i)}_\gamma,w^{(i,j)}_\gamma)$.
Thus, we have $V_\gamma^{(i,j)}(x) \le V_\gamma^{(i)}(x)$ as desired. In an analogous way,
utilizing that $w^{(i,0)}_\gamma=0$ we find
\begin{equation*}
	V_\gamma^{(i,0)}(x)=\int^\infty_0 \big(\ell(y(t))+ \|u^{(i)}_\gamma(y(t))\|^2_R\big)\, dt \ge 0 ,
\end{equation*}
where $y=y(u^{(i)}_\gamma,w=0)$. Moreover $V_\gamma^{(i,0)}(x)>0$ for $x\neq 0$.
\end{proof}

\begin{remark}\label{rm3}
{\em {
		Since the sequence $V^{(i,j)}_\gamma(x)$ is monotonically increasing and bounded above  for each $x\in \Omega$, it follows that there exists a function denoted by $V^{(i,\infty)}_\gamma$ such that $\lim_{j\to \infty}V^{(i,j)}_\gamma(x)=V^{(i,\infty)}_\gamma(x) $ for every $x\in\Omega$.
		
		If $\Omega$ is bounded and  $V^{(i,\infty)}_\gamma(x)$ is continuous on  $\bar\Omega$ then the convergence is uniform by Dini's theorem.  Alternatively, if $V^{(i,0)}_\gamma$ is uniformly  bounded from below on $\Omega$, then by Lebesgue's monotone convergence theorem $\{V^{(i,j)}_\gamma\}_{j=1}^\infty$ converges in $L^1(\Omega)$ to $V^{(i,\infty)}_\gamma$.
}}
\end{remark}

\begin{remark}\label{rm4}{\em	We can show that $\{w^{(i,j)}_\gamma\}$ converges to $w^{(i,\infty)}_\gamma$ on $\Omega$ under the following assumptions.
	If $\Omega$ is bounded and $\{V^{(i,j)}_{\gamma}\}\in C^1(\Omega)$ satisfy \eqref{eq:aux1}, $V^{(i,\infty)}_\gamma\in C^1(\Omega)\cap C(\bar\Omega)$, and $\{\nabla V^{(i,j)}_\gamma\}$ is equicontinuous in $\Omega$, then it can be shown that $\{\nabla V^{(i,j)}_{\gamma}\}_{j=1}^\infty$ converges pointwise to $\nabla V^{(i,\infty)}_{\gamma}$,
	$\{w^{(i,j)}_\gamma\}$ converges to $w^{(i,\infty)}_\gamma=\frac{1}{2\gamma^2}P^{-1}h(x)^t\nabla V^{(i,\infty)}_{\gamma}(x)$, and  $V^{(i,\infty)}_{\gamma}(x)$ solves the
	HJI equation
	\begin{align}\label{eqhj}
		\nabla V^{(i,\infty)}_{\gamma}(x)^t\big(f+gu^{(i)}_{\gamma}\big)+\frac{1}{4\gamma^2}\nabla V^{(i,\infty)}_{\gamma}(x)^t hP^{-1}h^t\nabla V^{(i,\infty)}_\gamma(x)+\norm{u_\gamma^{(i)}}^2_R+\ell(x)=0.
	\end{align}
	This can be argued similarly as  Proposition 2 in \cite{kk20}.
	Conversely, conditions for the  existence of a   solution $V^{(i,\infty)}_\gamma$ to \eqref{eqhj}, are given in \cite{V92}, for example.
}
\end{remark}

Next we provide sufficient conditions for the monotonicity for the sequence $V^{(i,\infty)}_\gamma$ which was introduced in Remark \ref{rm1}.
\begin{theorem}\label{lm2}
Let $V_\gamma^{(i,\infty)}\in C^1(\R^d)$ satisfy \eqref{eqhj} on $\R^d$,  and assume that
\begin{align}\label{eq1.5}
	\dot y= f(y)+gu^{(i)}_{\gamma}+hw^{(i,\infty)}_{\gamma}, \quad y(0)=x,
\end{align}
is asymptotically stable  on $\Omega$ for each $i$.
Then we have
$0\le  V_\gamma^{(i+1,\infty)}(x)\leq V_\gamma^{(i,\infty)}(x)$ $\forall x\in \Omega$.
\end{theorem}

\begin{remark}\label{rm5}
{\em As a consequence of the previous theorem,   the pointwise convergence of $V_\gamma^{(i,\infty)}$ follows.  The reason why the regularity of $V_\gamma^{(i,\infty)}$ and equation \eqref{eqhj} are assumed on all of $\R^d$ relates to the fact that the forced trajectories emanating from some $x\in \Omega$ need not remain in $\Omega$ in general. It could be an interesting issue of  further work to provide conditions on the problem data which imply a-priori bounds on these forced trajectories.
}
\end{remark}
\begin{proof}

Since \eqref{eq1.5} is asymptotically stable for each $i$, along the trajectory $\dot y= f+gu^{(i+1)}_\gamma+hw^{(i+1,\infty)}_\gamma$, we obtain the difference
between $V^{(i+1,\infty)}_\gamma(x)$ and $V^{(i,\infty)}_\gamma(x)$   as
\begin{align*}
	&V^{(i+1,\infty)}_\gamma(x)-V^{(i,\infty)}_\gamma(x)\\
	&=\int_{0}^{\infty}\Bigg(\Big({\nabla V_\gamma^{(i,\infty)}}^t(f+gu^{(i+1)}_\gamma+hw^{(i+1,\infty)}_\gamma)\Big)
	-\Big({\nabla V^{(i+1,\infty)}_\gamma}^t(f+gu^{(i+1)}_\gamma+hw^{(i+1,\infty)}_\gamma)\Big)\Bigg)\; dt.
\end{align*}
From the HJI equation \eqref{eqhj}, it follows that for $i\geq -1$
\begin{align}\label{ex2.2}
	\nabla {V^{(i+1,\infty)}_{\gamma}}^t\Big(f+gu^{(i+1)}_{\gamma}+hw^{(i+1,\infty)}_{\gamma}\Big)
	=-\Big(\ell(y)+\norm{u^{(i+1)}_{\gamma}}^2_R-\gamma^2\norm{w^{(i+1,\infty)}_{\gamma}}^2_P\Big).
\end{align}
To estimate ${\nabla V^{(i,\infty)}_{\gamma}}^t(f+gu^{(i+1)}_{\gamma}+hw^{(i+1,\infty)}_{\gamma})$, we rewrite it as
\begin{align*}
	\nabla {V^{(i,\infty)}_{\gamma}(y)}^t&\Big(f+gu^{(i+1)}_{\gamma}+hw^{(i+1,\infty)}_{\gamma}\Big)\\
	&=\nabla {V^{(i,\infty)}_{\gamma}(y)}^t\Big(f+gu^{(i)}_{\gamma}+hw^{(i,\infty)}_{\gamma}\Big)+\langle h^t\nabla {V^{(i,\infty)}_{\gamma}(y)}, w^{(i+1,\infty)}_{\gamma}-w^{(i,\infty)}_{\gamma}\rangle\\
	&\quad+\langle g^t\nabla {V^{(i,\infty)}_{\gamma}(y)}, u^{(i+1)}_{\gamma}-u^{(i)}_{\gamma}\rangle.
\end{align*}
With the notation of $z=\frac{1}{2}R^{-1}g^t\nabla V^{(i,\infty)}_{\gamma}$  we obtain $g^t\nabla V^{(i,\infty)}_{\gamma}=2Rz$ and $u^{(i+1)}_\gamma=\cP_{U}(-z)$.
Also from the worst disturbance, it follows that $h^t\nabla V^{(i,\infty)}_{\gamma}=2\gamma^2Pw^{(i,\infty)}_{\gamma}$.
Altogether, we have
\begin{align}\label{e2.5}
	&\nabla {V^{(i,\infty)}_{\gamma}(y)}^t\Big(f(y)+g(y)u^{(i+1)}_{\gamma}(y)+h(y)w^{(i+1,\infty)}_{\gamma}(y)\Big)\notag\\
	&=-\ell(y)-\gamma^2\norm{w^{(i,\infty)}_{\gamma}}^2_P+2\gamma^2P\langle w^{(i,\infty)}_{\gamma},w^{(i+1,\infty)}_{\gamma}\rangle\notag\\
	&\quad-\norm{u^{(i)}_{\gamma}}^2_R+\langle 2Rz,\cP_{U}(-z)-u^{(i)}_\gamma\rangle\notag\\
	&=-\ell(y)+\gamma^2\norm{w^{(i+1,\infty)}_{\gamma}}^2_P-\gamma^2\norm{w^{(i+1,\infty)}_{\gamma}-w^{(i,\infty)}_{\gamma}}^2_P- \|\cP_{U}(-z) \|_R^2\\
	&\quad - \|u^{(i)}_\gamma(y) - \cP_{U}(-z) \|_R^2+ 2(z+ \cP_{U}(-z))^tR(\cP_{U}(-z)-u^{(i)}_\gamma(y))\notag.
\end{align}
With these expressions we obtain
\begin{align*}
	&V^{(i+1,\infty)}_\gamma(x)-V^{(i,\infty)}_\gamma(x)\\
	&=\int_{0}^{\infty}\Bigg(\Big(-\|u^{(i)}_\gamma(y) - \cP_{U}(-z) \|_R^2-\gamma^2\norm{w^{(i+1,\infty)}_{\gamma}-w^{(i,\infty)}_{\gamma}}^2_P\\
	&\qquad + 2(z+ \cP_{U}(-z))^tR(\cP_{U}(-z)-u^{(i)}_\gamma(y))\Big)\Bigg)\; dt.
\end{align*}
Since $(z+ \cP_{U}(-z))^tR(\cP_{U}(-z)-u^{(i)}_\gamma(y))\leq 0$, we get
$V^{(i+1,\infty)}_\gamma(x)\leq V^{(i,\infty)}_\gamma(x)$.

The lower bound $0\le  V^{(i,\infty)}_\gamma(x)$ is a consequence of Theorem \ref{hlm2.1}.
\end{proof}
\begin{remark}\label{rm6}
{\em If $\Omega$ is bounded and $\{V^{(i,\infty)}_{\gamma}\}\in C^1(\Omega)$ satisfy \eqref{eqhj}, $V_\gamma\in C^1(\Omega)\cap C(\bar\Omega)$, and $\{\nabla V^{(i,\infty)}_\gamma\}$ is equicontinuous in $\Omega$, then the pair $(u^{(i)}_{\gamma}, V^{(i,\infty)}_\gamma)$ of control and value function converges to $\Big(u^*=\cP_{U}\Big(-\frac{1}{2}R^{-1}g(x)^t\nabla V_\gamma(x)\Big), V_{\gamma}(x)\Big)$, where $V_{\gamma}(x)$ solves the HJI equation \eqref{hji2}. For a proof, we refer to \cite{kk20}.
}\end{remark}

At last we show the asymptotic stability property of $u^{(i)}_\gamma$ on $\Omega$ for the unperturbed trajectory at each iteration level. Note that $i=0$ is assumed in the presentation of Algorithm \ref{alg:sg1}.

\begin{lemma}\label{xlm2}
If   $V_\gamma^{(i,\infty)}\in C^1(\R^d)$ for $i\ge 0$ and  \eqref{eqhj} holds,
then function $u_\gamma^{(i+1)}$
is an asymptotically  stabilizing control on $\Omega$  for the unperturbed system $\dot y=f(y)+gu$.
\end{lemma}
\begin{proof}
We show that $V^{(i,\infty)}_\gamma$ is a strict Lyapunov function for the system \eqref{eq1.2} with $w=0$. By Theorem \ref{hlm2.1} the function  $V^{(i,\infty)}_\gamma$ is positive definite.
Next we take the time derivative of $t\to V^{(i,\infty)}_{\gamma}(y)(t)$ along the trajectory $\dot y=f(y)+gu^{(i+1)}_{\gamma}$. It  can be expressed as
\begin{align*}
	\frac{d}{dt} V^{(i,\infty)}_{\gamma}(y(t))&=\nabla {V^{(i,\infty)}_{\gamma}(y)}^t\Big(f+gu^{(i+1)}_{\gamma}\Big)\\
	&=\nabla {V^{(i,\infty)}_{\gamma}(y)}^t\Big(f+gu^{(i)}_{\gamma}\Big)+\langle g^t\nabla {V^{(i,\infty)}_{\gamma}(y)}, u^{(i+1)}_{\gamma}-u^{(i)}_{\gamma}\rangle.
\end{align*}
Using \eqref{eqhj}, we obtain
\begin{align}\label{eq1.8}
	\frac{d}{dt} V^{(i,\infty)}_{\gamma}(y(t))
	&=-\ell(y)-\frac{1}{4\gamma^2}\nabla V^{(i,\infty)}_{\gamma}(y)^t hP^{-1}h^t\nabla V^{(i,\infty)}_\gamma(y)-\norm{u_\gamma^{(i)}}^2_R\notag\\
	&\qquad+\langle g^t\nabla {V^{(i,\infty)}_{\gamma}(y)}, u^{(i+1)}_{\gamma}-u^{(i)}_{\gamma}\rangle.
\end{align}
Let $z=\frac{1}{2}R^{-1}g^t\nabla V^{(i,\infty)}_{\gamma}$, so that we have $g^t\nabla V^{(i,\infty)}_{\gamma}=2Rz$. Consequently, from the control law \eqref{eq:aux3} we obtain $u^{(i+1)}_\gamma=\cP_{U}(-z)$.
Hence, from \eqref{eq1.8} we obtain
\begin{align*}
	\frac{d}{dt} V^{(i,\infty)}_{\gamma}(y(t))&=\nabla {V^{(i,\infty)}_{\gamma}(y)}^t\Big(f+gu^{(i+1)}_{\gamma}\Big)\\
	&=-\ell(y)-\frac{1}{4\gamma^2}\nabla V^{(i,\infty)}_{\gamma}(y)^t hP^{-1}h^t\nabla V^{(i,\infty)}_\gamma(y)-\|\cP_{U}(-z) \|_R^2\\
	&\quad - \|u^{(i)}_\gamma(y) - \cP_{U}(-z) \|_R^2+ 2(z+ \cP_{U}(-z))^tR(\cP_{U}(-z)-u^{(i)}_\gamma(y)).
\end{align*}
As $u^{(i)}_\gamma(y)\in U$, it follows that  $(z+ \cP_{U}(-z))^tR(\cP_{U}(-z)-u^{(i)}_\gamma(y)) \leq 0$, and therefore we obtain finally
$\frac{d}{dt} V^{(i,\infty)}_{\gamma}(y(t))=\nabla {V^{(i,\infty)}_{\gamma}(y)}^t\Big(f+gu^{(i+1)}_{\gamma}\Big)<0$, and the asymptotic stabilization property of the control follows.
\end{proof}

\section{$\cH_{\infty}$  feedback control for stochastic systems}
In this section, we consider the following infinite horizon stochastic optimal robust control problem:
\begin{align}\label{seq1.1}
\underset{u(\cdot)\in \cU}{\min}\;\underset{w(\cdot)\in \cW}{\max}\;&\cJ(u,w;x):=\E\int\limits_0^\infty \Big(\ell(y(t))+\|u(t)\|_R^2-\gamma^2\|w(t)\|_P^2\Big)\, dt,
\end{align}
subject to the nonlinear stochastic system
\begin{align}\label{eq2.1}
d y&= \Big(f(y)+g(y)u(t)+h(y)w(t)\Big)dt+\Big(g_1(y)+g_2(y)w(t)\Big)dW \,,\quad y(0)=x,
\end{align}
where $y(t)=(y_1(t),\ldots,y_d(t))^t\in\R^d$ is the state vector,  $W(t)\in \R^k$
is a standard $k$-dimensional Wiener process (i.e., with independent components)
defined on a complete probability space. 
The exogenous disturbance $w(\cdot)\in\cW=\cL^2_{\cF}(\R_+;\R^q)$ and the control input
	$u(\cdot)\in\cU=\cL^2_{\cF}(\R_+;U)$, where the subscript $\cF$ indicates that the processes are
	adapted to the natural filtration $\{\cF_t\}$ generated by $W$ and square integrable,
	$\E\int_0^\infty\|\cdot\|^2\,dt<\infty$, see \cite{ZC06}. Since we are interested in closed loop
	control, we note that the strong solution $y$ of \eqref{eq2.1} is $\{\cF_t\}$-adapted, so that
	feedback laws $u(y(t))$, $w(y(t))$ are admissible.
	Further $U$  is a closed convex set containing the origin.
	
	Moreover, for the dynamics, which are characterized by $f$, $h$, $g$  (mentioned on page 3), $g_1(y):\R^d\rightarrow \mathbb{R}^{d\times k}$ and
	$g_2(y)=[g_2^1(y),\dots,g_2^q(y)]$ with $g_2^l(y):\mathbb{R}^d\rightarrow\mathbb{R}^{d\times k}$, where
	$g_2(y)w:=\sum_{l=1}^q w_l\,g_2^l(y)$ for $w\in\mathbb{R}^q$, it is assumed that these functions are Lipschitz continuous on $\R^d$ and that $g$, $h$, $g_1$ and $g_2$ are globally bounded. The analysis is carried out over a  bounded domain  $\Omega\subset \R^d$  containing the origin. Concerning the existence and uniqueness results for \eqref{eq2.1}, we refer to  e.g. \cite[Chapter 2]{M07}.
	Throughout we assume that  $f(0)=0$, $g_1(0)=0$ for $w=0$ and $u=0$.\\
	
	We want  to find a control $u\in\cL^2_{\mathcal{F}}((0,T),U)$
	such that the stochastic system \eqref{eq2.1} is closed loop asymptotically stable in probability and has an $\cL_2$-gain less than or equal to the disturbance attenuation level $\gamma >0$, i.e. for all $T\geq 0$ and $w\in\cL^2_{\mathcal{F}}((0,T),\mathbb{R}^q)$,
	\begin{equation}\label{eq2.2}
\E\int_{0}^{T}\Big(\ell(y)+\norm{u(t)}^2_{R}\Big)\; dt\leq \gamma^2\, \E\int_{0}^{T}\norm{w(t)}^2_{P}\;dt,
\end{equation}
where $y$ is the solution to \eqref{eq2.1} with $x=0$.
We define the value function
$$
V_\gamma(x)=\underset{u(\cdot)\in \cU}{\min}\;\underset{w(\cdot)\in \cW}{\max}\;\cJ(u,w;x).$$
Here, for $V_\gamma\in C^2(\R^d)$, the traces involving $g_2$ are taken columnwise, i.e.
	$Tr[g_2^t\frac{\partial^2V_\gamma}{\partial x^2}g_2]
	:=\big(Tr[(g_2^l)^t\frac{\partial^2V_\gamma}{\partial x^2}g_2^m]\big)_{l,m=1}^q\in\R^{q\times q}$,
	$Tr[g_2^t\frac{\partial^2V_\gamma}{\partial x^2}g_1]
	:=\big(Tr[(g_2^l)^t\frac{\partial^2V_\gamma}{\partial x^2}g_1]\big)_{l=1}^q\in\R^q$ and
	$Tr[g_1^t\frac{\partial^2V_\gamma}{\partial x^2}g_2]:=Tr[g_2^t\frac{\partial^2V_\gamma}{\partial x^2}g_1]^t$,
	and $2\gamma^2P-Tr[g_2^t\frac{\partial^2V_\gamma}{\partial x^2}g_2]>0$ is understood in the sense of
	positive definite matrices.
	\begin{assumption}\label{ass2}
The radially unbounded value function satisfies  $V_\gamma\in C^2(\R^d)$.  Further, for the policy iteration Algorithm \ref{alg:sg2}, we assume that $V^{(i,j)}_{\gamma}(x)\in C^2(\R^d)$ for each $i,\hspace{0.1cm} j\geq 0$.
\end{assumption}

Similar assumptions were made  earlier in \cite{WS02, ZC06}.
If in addition $V_\gamma\ge 0$ satisfies the  HJI equation
\begin{equation}\label{eq2.3}
\Bigggg\{\begin{array}{r@{}l}
	\nabla V_\gamma(x)^t&\Big(f(x)+g(x)\cP_{U}\Big(-\frac{1}{2}R^{-1}g^t\nabla V_\gamma(x)\Big)\Big)\\
	&+ \frac{1}{2}\Big(\nabla V_\gamma(x)^th+Tr[g_1^t\frac{\partial^2 V_\gamma}{\partial x^2}g_2]\Big)(2\gamma^2 P-Tr[g_2^t\frac{\partial^2 V_\gamma}{\partial x^2}g_2])^{-1}\Big(h^t\nabla V_\gamma(x)+Tr[g_2^t\frac{\partial^2 V_\gamma}{\partial x^2}g_1]\Big)\\
	&\quad+\norm{\cP_{U}\Big(-\frac{1}{2}R^{-1}g^t\nabla V_\gamma(x)\Big)}^2_R+\ell(x)+\frac{1}{2}Tr[g_1^t\frac{\partial^2 V_\gamma}{\partial x^2}g_1]=0,\\
	& V_\gamma(0)=0, \quad 2\gamma^2P-Tr[g_2^t\frac{\partial^2 V_\gamma}{\partial x^2}g_2]>0,
\end{array}
\end{equation}
then system \eqref{eq2.1} has a  $\cL_2$-gain less than or equal to $\gamma>0$,  with  optimal control  and optimal disturbance strategies given by
$$u^*_\gamma(x)\!=\!\cP_{U}\Big(\!-\frac12R^{-1}g(x)^t\nabla V_\gamma(x)\!\Big) \text{ and }
w^*_{\gamma}(x)\!=\!\Big(\!2\gamma^2P-Tr[g_2^t\frac{\partial^2 V_\gamma}{\partial x^2}g_2]\Big)^{-1}\!\Big(\!h^t\nabla V_{\gamma}(x)+Tr[g_2^t\frac{\partial^2 V_\gamma}{\partial x^2}g_1]\Big).
$$
For the case without control constraints, we refer to \cite[ Theorem 3.1, Remark 3.1]{ZC06}.
The proofs there can be adjusted to handle the control constrained case which we are treating.
In the remainder of this section we simply write $(u_\gamma,w_{\gamma})$ for  $(u^*_\gamma,w^*_\gamma)$. Moreover, when there is no confusion, we suppress the dependence of $f\;,g\;,h\;,g_1$ and $g_2$ on the state variable.

The  HJI equation related to unconstrained control satisfies
\begin{equation}\label{eqx2.3}
\Bigggg\{\begin{array}{r@{}l}	
	\nabla V_\gamma(x)^t&f(x)-\frac{1}{4}\nabla V_\gamma(x)^tg(x)R^{-1}g(x)^t\nabla V_\gamma(x)\\
	&+ \frac{1}{2}\Big(\nabla V_\gamma(x)^th+Tr[g_1^t\frac{\partial^2 V_\gamma}{\partial x^2}g_2]\Big)(2\gamma^2 P-Tr[g_2^t\frac{\partial^2 V_\gamma}{\partial x^2}g_2])^{-1}\Big(h^t\nabla V_\gamma(x)+Tr[g_2^t\frac{\partial^2 V_\gamma}{\partial x^2}g_1]\Big)\\
	&\qquad+\ell(x)+\frac{1}{2}Tr[g_1^t\frac{\partial^2 V_\gamma}{\partial x^2}g_1]=0\,,\\
	V_\gamma(0)=&0\,,\quad 2\gamma^2P-Tr[g_2^t\frac{\partial^2 V_\gamma}{\partial x^2}g_2]>0.
\end{array}
\end{equation}


For $h=0=g_2$ in \eqref{eq2.3}-\eqref{eqx2.3}, we obtain the corresponding constrained and unconstrained 
HJB equation which arise from stochastic control of dynamical systems
respectively.
The infinitesimal generator of \eqref{eq2.1} with optimal pair $(u_\gamma,w_\gamma)$ is


$\cL_{(u_\gamma,w_\gamma)}V_\gamma=\nabla V_\gamma^t\Big(f+gu_\gamma+hw_\gamma\Big)+\frac{1}{2}Tr[(g_1+g_2w_\gamma)^t\frac{\partial^2 V_\gamma}{\partial x^2}(g_1+g_2w_\gamma)]$.

%
%
%

For the stochastic system \eqref{eq2.1}, $y=0$ is said to be stable in probability, or stochastically stable in the sense of Lyapunov, if for any $\epsilon>0$, $\lim_{x\to 0}P(\sup_{t\geq 0}\norm{y(t)}>\epsilon)=0$.  Further  $y=0$ in \eqref{eq2.1}, is said to be globally asymptotically stable in the sense of Lyapunov in probability, if it is stable in probability and for all $x\in \R^d$, $P(\lim_{t\to \infty} y(t)=0)=1$. A sufficient condition for asymptotically stability in probability in the sense of Lyapunov is as follows: If there exists a positive definite decrescent Lyapunov function $V_\gamma \in C^2(\R^d)\;$ satisfying $\cL_{(u_\gamma,w_\gamma)}V_\gamma<0$ for all non-zero $y\in \R^d$, then the equilibrium point $0$ of \eqref{eq2.1} is asymptotically stable in probability \cite{LH24, M07}

The second order  GHJI equations is of the form
\begin{align}
\Bigg\{\begin{array}{r@{}l}
	GHJI(V_\gamma,\nabla V_\gamma,\nabla^2V_\gamma,u_\gamma;w_\gamma)&=0,\qquad V_\gamma(0)=0,\quad 2\gamma^2P-Tr[g_2^t\frac{\partial^2 V_\gamma}{\partial x^2}g_2]>0 \quad \label{gHJI} \text{where}\\
	GHJI(V_\gamma,\nabla V_\gamma,\nabla^2V_\gamma,u_\gamma;w_\gamma)&:=\cL_{(u_\gamma,w_\gamma)}V_\gamma+\Big(\ell+\norm{u_\gamma}^2_{R}-\gamma^2\norm{w_\gamma}^2_{P}\Big),
\end{array}
\end{align}
and $\nabla V_\gamma(x)=(\px{1}V_\gamma,\ldots,\px{d}V_\gamma)^t$.
Below we discuss policy iteration  where we calculate $w_\gamma$ and $u_\gamma$ iteratively.
\begin{algorithm}[!ht]
Let $u^{(0)}_\gamma$ be an asymptotically stabilizing control law  for the system \eqref{eq2.1} with $w_\gamma \equiv 0$ and $\gamma\ge \gamma^*$.
\begin{algorithmic}
	\STATE{\bf For $i=0$ to $\infty$\;}
	\STATE{\bf\qquad Set $w^{(i,0)}_{\gamma}(x)\equiv 0$,\;}
	\STATE{\bf\qquad For $j=0$ to $\infty$\;}
	\STATE{\bf\hspace{1cm} Solve for $V^{(i,j)}_{\gamma}(x)\in C^2(\R^d):$
		\begin{equation}\label{eq2.4}
			\Bigg\{\begin{array}{r@{}l}
				\cL_{(u^{(i)}_\gamma,w^{(i,j)}_\gamma)}V^{(i,j)}_\gamma(x)+ \Big(\ell+\norm{u^{(i)}_\gamma}^2_{R}-\gamma^2\norm{w^{(i,j)}_{\gamma}}_P^2\Big)=0\,,\\
				V^{(i,j)}_\gamma(0)=0\,,\quad 2\gamma^2P-Tr[g_2^t\frac{\partial^2 V^{(i,j)}_\gamma}{\partial x^2}g_2]>0.
			\end{array}
		\end{equation}
	}
	\vspace{-3mm}
	\STATE{\bf \hspace{1cm} Update the Disturbance:
		\[
		w^{(i,j+1)}_{\gamma}(x)=\Big(2\gamma^2P-Tr[g_2^t\frac{\partial^2 V^{(i,j)}_\gamma}{\partial x^2}g_2]\Big)^{-1}\Big(h^t\nabla V^{(i,j)}_{\gamma}(x)+Tr[g_2^t\frac{\partial^2 V^{(i,j)}_\gamma}{\partial x^2}g_1]\Big)\,,
		\]
		\qquad	j=j+1	\;}
	
	\bf {\qquad End j loop
		
		\STATE{Update the Control:
			\vspace{-3mm}
			\[
			u^{(i+1)}_{\gamma}(x)=\cP_{U} \Big(-\frac{1}{2}R^{-1}g^t\nabla V_\gamma^{(i,\infty)}(x))\Big)\,.
			\]
			i=i+1	\;}
		
		End i loop}
	\caption{Policy iteration Algorithm}\label{alg:sg2}
\end{algorithmic}
\end{algorithm}

The following theorem establishes convergence related to the value function of policy iteration for the inner loop Algorithm \ref{alg:sg2}. To obtain the upper bound for $V^{(i,j)}_\gamma$, similarly as in Theorem \ref{hlm2.1}, we define for  a stabilizing control $u_\gamma^{(i)}$ and $x\in\Omega$
\begin{equation}\label{eqkk2}
\left\{
\begin{array}{ll}
	V^{(i)}_\gamma(x):=\max_{w(\cdot)\in \cW} \cJ(u_\gamma^{(i)},w;x)\\[1.7ex]
	d y= \Big(f(y(t))+g(y)u^{(i)}_{\gamma}(y(t))+h(y)w(t)\Big) dt+\Big(g_1(y)+g_2(y)w(t)\Big)dW\,,\quad y(0)=x.
\end{array}
\right.
\end{equation}
\begin{theorem}\label{lm2.1}
If  the stochastic system
\begin{align}\label{eq2.8}
	d y&= \Big(f(y)+g(y)u^{(i)}_\gamma+h(y)w^{(i,j)}_\gamma\Big)\,dt+\Big(g_1(y)+g_2(y)w^{(i,j)}_\gamma\Big)\,dW, \quad y(0)=x,
\end{align}
is asymptotically stable in the sense of Lyapunov on $\Omega$ for all pairs $(i,j)$, where $i\geq 0$ is fixed and all $j\geq 0$ and $V_\gamma^{(i,j)}\in C^2(\bar \Omega)$, then we get
$V^{(i,j)}_\gamma(x)\leq V^{(i,j+1)}_\gamma(x)\leq V^{(i)}_\gamma(x)$ for each $x\in \Omega$. Further $V^{(i,j)}_\gamma(x)>0$ for $x\neq 0$.
\end{theorem}
\begin{proof}
We first  calculate the difference   $V^{(i,j+1)}_\gamma-V^{(i,j)}_\gamma$ by It\^{o}'s formula along the trajectory
\begin{align}\label{eq2.7}
	d y&= \Big(f(y)+g(y)u^{(i)}_\gamma+h(y)w^{(i,j+1)}_\gamma\Big)\,dt+\Big(g_1(y)+g_2(y)w^{(i,j+1)}_\gamma\Big)\,dW.
\end{align}
Applying It\^{o}'s formula \cite{LH24} to  $V_\gamma^{(i,j+1)}(y)$ over a finite time interval $[0,T]$, we get
\begin{align}\label{xeq}
	V_\gamma^{(i,j+1)}(y(T)) - V_\gamma^{(i,j+1)}(y(0))
	=\int_0^T \cL_{(u^{(i)}_\gamma,w^{(i,j+1)}_\gamma)}V_\gamma^{(i,j+1)}(y)\; dt +M_T \;,
\end{align}
where $M_T$ is the stochastic integral term $$	M_T = \int_0^T \nabla V_\gamma^{(i,j+1)}(y(t))^t \left(g_1(y(t))+g_2(y(t))w^{(i,j+1)}_\gamma\right) dW(t).$$	
By Assumption \ref{ass2}, $\nabla V_\gamma^{(i,j+1)}(y)$ and $\frac{\partial^2 V^{(i,j)}_\gamma}{\partial y^2}$
are  bounded on $\bar \Omega$, and consequently on $\Omega$. Further, $h(y)$, $g_1(y)$ and $g_2(y)$ are bounded on $\Omega$ by the assumption on the dynamics. As a consequence                           $\left(g_1(y(t))+g_2(y(t))w^{(i,j+1)}_\gamma\right)$  is also bounded on $\Omega$.  Hence $M_T$ is a true martingale. This implies $\E[M_T] = 0$ for any finite $T$. Hence we obtain from \eqref{xeq}
$$	\E V_\gamma^{(i,j+1)}(y(T)) - V_\gamma^{(i,j+1)}(x) = \mathbb{E}\left[\int_0^T \cL_{(u^{(i)}_\gamma,w^{(i,j+1)}_\gamma)}V_\gamma^{(i,j+1)}(y)\; dt\right]. $$
Then,  taking the limit as $T \to \infty$, we obtain $-V_\gamma^{(i,j+1)}(x)=\mathbb{E}\left[\int_0^\infty \cL_{(u^{(i)}_\gamma,w^{(i,j+1)}_\gamma)}V_\gamma^{(i,j+1)}(y)\; dt\right] $. 

Since \eqref{eq2.8} is asymptotically stable on $\Omega$ for each pair $(i,j)$,
by It\^{o}'s formula we find
\begin{align*}
	&V^{(i,j)}_\gamma(x)-V^{(i,j+1)}_\gamma(x)\\
	&=\E\int_{0}^{\infty}\Bigg(\Big({\nabla V^{(i,j+1)}_\gamma}^t(f+gu^{(i)}_\gamma+hw_\gamma^{(i,j+1)})+\frac{1}{2}Tr[(g_1+g_2w_\gamma^{(i,j+1)})^t\frac{\partial^2 V^{(i,j+1)}_\gamma}{\partial y^2}(g_1+g_2w_\gamma^{(i,j+1)})]\Big)\\
	&\quad -\Big({\nabla V^{(i,j)}_\gamma}^t(f+gu^{(i)}_\gamma+hw_\gamma^{(i,j+1)})+\frac{1}{2}Tr[(g_1+g_2w_\gamma^{(i,j+1)})^t\frac{\partial^2 V^{(i,j)}_\gamma}{\partial y^2}(g_1+g_2w_\gamma^{(i,j+1)})]\Big)\Bigg) dt\\
	&\quad =: \E \int_{0}^{\infty}(A-B)\; dt.
\end{align*}
Using the GHJI equation \eqref{eq2.4}, it follows that
$A=-\Big(\ell(y)+\norm{u^{(i)}_\gamma}^2_R-\gamma^2\norm{w^{(i,j+1)}_\gamma}^2_P\Big)$.
Applying the update disturbance $w^{(i,j+1)}_\gamma$ i.e. using
\begin{align*}
	-\Big(2\gamma^2P-Tr[g_2^t\frac{\partial^2 V^{(i,j)}_\gamma}{\partial y^2}g_2]\Big)w^{(i,j+1)}_\gamma+Tr[g_2^t\frac{\partial^2 V^{(i,j)}_\gamma}{\partial y^2}g_1]=-h^t\nabla V^{(i,j)}_{\gamma}(y),
\end{align*}
we get
\begin{align}\label{eq3.1}
	&\nabla V^{(i,j)}_{\gamma}(y)^t\big(f+hw^{(i,j+1)}_{\gamma}+g u^{(i)}_\gamma\big)+\frac{1}{2}Tr[(g_1+g_2w^{(i,j+1)}_\gamma)^t\frac{\partial^2 V^{(i,j)}_\gamma}{\partial y^2}(g_1+g_2w^{(i,j+1)}_\gamma)]\notag\\
	&=-\Big(\ell(y)+\norm{u^{(i)}_\gamma}^2_R-\gamma^2\norm{w^{(i,j)}_\gamma}^2_P\Big)+\langle h^t\nabla V^{(i,j)}_\gamma,w^{(i,j+1)}_\gamma-w^{(i,j)}_\gamma\rangle\notag\\
	&\qquad+\frac{1}{2}Tr[(g_1+g_2w^{(i,j+1)}_\gamma)^t\frac{\partial^2 V^{(i,j)}_\gamma}{\partial y^2}(g_1+g_2w^{(i,j+1)}_\gamma)]\\
	&\qquad-\frac{1}{2}Tr[(g_1+g_2w^{(i,j)}_\gamma)^t\frac{\partial^2 V^{(i,j)}_\gamma}{\partial y^2}(g_1+g_2w^{(i,j)}_\gamma)]\notag\\
	&\hspace{1cm}=-\Big(\ell(y)+\norm{u^{(i)}_\gamma}^2_R-\gamma^2\norm{w^{(i,j)}_\gamma}^2_P\Big)+\Big\langle(2\gamma^2P-Tr[g_2^t\frac{\partial^2 V^{(i,j)}_\gamma}{\partial y^2}g_2])w^{(i,j+1)}_\gamma,\notag\\
	&\hspace{1cm} w^{(i,j+1)}_\gamma-w^{(i,j)}_\gamma\Big\rangle+\Big\langle\frac{1}{2}Tr[g_2^t\frac{\partial^2 V^{(i,j)}_\gamma}{\partial y^2}g_2](w^{(i,j+1)}_\gamma+w^{(i,j)}_\gamma), w^{(i,j+1)}_\gamma-w^{(i,j)}_\gamma\Big\rangle\notag\\
	&=-\Big(\ell(y)+\norm{u^{(i)}_\gamma}^2_R-\gamma^2\norm{w^{(i,j)}_\gamma}^2_P\Big)+\langle 2\gamma^2Pw^{(i,j+1)}_\gamma, w^{(i,j+1)}_\gamma-w^{(i,j)}_\gamma\rangle\notag\\
	&\quad -\tfrac12\big(w^{(i,j+1)}_\gamma-w^{(i,j)}_\gamma\big)^t
			Tr\big[g_2^t\frac{\partial^2V^{(i,j)}_\gamma}{\partial y^2}g_2\big]
			\big(w^{(i,j+1)}_\gamma-w^{(i,j)}_\gamma\big)\notag.
\end{align}
Hence, an estimate for $B$ is obtained from \eqref{eq3.1}.
Altogether, we arrive at
$$A-B=-\big(w^{(i,j+1)}_\gamma-w^{(i,j)}_\gamma\big)^t
		\Big(\gamma^2P-\tfrac12Tr\big[g_2^t\frac{\partial^2V^{(i,j)}_\gamma}{\partial y^2}g_2\big]\Big)
		\big(w^{(i,j+1)}_\gamma-w^{(i,j)}_\gamma\big).$$
Therefore, it follows that
\begin{align*}
	&V^{(i,j)}_\gamma(x)-V^{(i,j+1)}_\gamma(x)\\
	&\quad=-\tfrac12\,\E\int_0^\infty
	\big(w^{(i,j+1)}_\gamma-w^{(i,j)}_\gamma\big)^t
			\Big(2\gamma^2P-Tr\big[g_2^t\frac{\partial^2V^{(i,j)}_\gamma}{\partial x^2}g_2\big]\Big)
			\big(w^{(i,j+1)}_\gamma-w^{(i,j)}_\gamma\big)\,dt\le0,
\end{align*}
where \eqref{gHJI}  was used. Consequently,  $V^{(i,j)}_\gamma(x)\leq V^{(i,j+1)}_\gamma(x)$.
The rest of the proof follows similarly as in Theorem \ref{hlm2.1}.
\end{proof}
\begin{remark}
If $V^{(i,j)}_\gamma\in C^2(\Omega)$ satisfy \eqref{eq2.4}, $V^{(i,\infty)}_\gamma\in C^2(\Omega)\cap C^1(\bar\Omega)$ and $\{\nabla^m V^{(i,j)}_\gamma\}$ is equicontinuous in $\Omega$, then we have $\{\nabla^m V^{(i,j)}_{\gamma}\}$ converges pointwise to $\nabla^m V^{(i,\infty)}_{\gamma}$ for $m=1,\,2$. Further, assuming the boundedness of $\Big(2\gamma^2P-Tr[g_2^t\frac{\partial^2 V^{(i,j)}_\gamma}{\partial x^2}g_2]\Big)^{-1}$ and $\Big(2\gamma^2P-Tr[g_2^t\frac{\partial^2 V^{(i,\infty)}_\gamma}{\partial x^2}g_2]\Big)^{-1}$ on $\Omega$,  $w^{(i,j)}_\gamma$ converges to $w^{(i,\infty)}_\gamma$ as $j\to \infty$.
Moreover, $V^{(i,\infty)}_\gamma(x)$ solves the equation
\begin{align}\label{eq2.9}
	\Biggg\{\begin{array}{r@{}l}
		\nabla V_\gamma(x)^t(f(x)+gu^{(i)}_\gamma)&+ \frac{1}{2}\Big(\nabla V_\gamma(x)^th+Tr[g_1^t\frac{\partial^2 V_\gamma}{\partial x^2}g_2]\Big)(2\gamma^2 P-Tr[g_2^t\frac{\partial^2 V_\gamma}{\partial x^2}g_2])^{-1}\Big(h^t\nabla V_\gamma(x)\\
		&\quad+Tr[g_2^t\frac{\partial^2 V_\gamma}{\partial x^2}g_1]\Big)+\norm{u^{(i)}_\gamma}^2_R+\ell(x)+\frac{1}{2}Tr[g_1^t\frac{\partial^2 V_\gamma}{\partial x^2}g_1]=0,\\
		&2\gamma^2P-Tr[g_2^t\frac{\partial^2 V_\gamma}{\partial x^2}g_2]>0,
	\end{array}
\end{align}
and
\begin{align}\label{worst}
	w^{(i,\infty)}_\gamma=\Big(2\gamma^2P-Tr[g_2^t\frac{\partial^2 V^{(i,\infty)}_\gamma}{\partial x^2}g_2]\Big)^{-1}\Big(h^t\nabla V^{(i,\infty)}_{\gamma}(x)+Tr[g_2^t\frac{\partial^2 V^{(i,\infty)}_\gamma}{\partial x^2}g_1]\Big).
\end{align}
\end{remark}
\begin{proof}
The proof follows similarly from Proposition 4 in \cite{kk20}.
\end{proof}
\begin{theorem}\label{lm2.2}
If the stochastic system
\begin{align}\label{eq2.10}
	d y&= \Big(f(y)+g(y)u^{(i)}_\gamma+h(y)w^{(i,\infty)}_\gamma\Big)\,dt+\Big(g_1(y)+g_2(y)w^{(i,\infty)}_\gamma\Big)\,dW ,
\end{align}
is asymptotically stable in probability on $\Omega$  for all $i$, and $V^{(i,\infty)}_\gamma(x)\in C^2(\R^d)$, then
$0\leq V^{(i+1,\infty)}_\gamma(x)\leq V^{(i,\infty)}_\gamma(x)$ holds for each $x\in \Omega$.
\end{theorem}
\begin{proof}
Due to  asymptotical stability of \eqref{eq2.10} for each $i$, along the trajectory $d y= \Big(f(y)+g(y)u^{(i+1)}_\gamma+h(y)w^{(i+1,\infty)}_\gamma\Big)\,dt+\Big(g_1(y)+g_2(y)w^{(i+1,\infty)}_\gamma\Big)\,dW$, we obtain the difference
between $V^{(i+1,\infty)}_\gamma(x)$ and $V^{(i,\infty)}_\gamma(x)$   as
\begin{align*}
	&V^{(i+1,\infty)}_\gamma(x)-V^{(i,\infty)}_\gamma(x)\\
	&=-\E\int_{0}^{\infty}\Bigg(\Big({\nabla V^{(i+1,\infty)}_\gamma}^t
	(f+gu^{(i+1)}_\gamma+hw^{(i+1,\infty)}_\gamma)+\frac{1}{2}Tr[(g_1+g_2w^{(i+1,\infty)}_\gamma)^t\frac{\partial^2 V^{(i+1,\infty)}_\gamma}{\partial y^2}\\
	&\qquad(g_1+g_2w^{(i+1,\infty)}_\gamma)]\Big)-\Big({\nabla V^{(i,\infty)}_\gamma}^t
	(f+gu^{(i+1)}_\gamma+hw^{(i+1,\infty)}_\gamma)\\
	&\hspace{2cm}+\frac{1}{2}Tr[(g_1+g_2w^{(i+1,\infty)}_\gamma)^t\frac{\partial^2 V^{(i,\infty)}_\gamma}{\partial y^2}(g_1+g_2w^{(i+1,\infty)}_\gamma)]\Big)\Bigg) \;dt.
	=: -\E \int_{0}^{\infty}(A-B)\; dt.
\end{align*}
Using the HJI equation \eqref{eq2.9}, it follows that
$A=-\Big(\ell(y)+\norm{u^{(i+1)}_\gamma}^2_R-\gamma^2\norm{w^{(i+1,\infty)}_\gamma}^2_P\Big)$, where $w^{(i+1,\infty)}_\gamma$ satisfies \eqref{worst}.
To obtain an estimate of $B$, we express $B$ as follows
\begin{align*}
	&{\nabla V^{(i,\infty)}_\gamma}^t
	(f+gu^{(i+1)}_\gamma+hw^{(i+1,\infty)}_\gamma)+\frac{1}{2}Tr[(g_1+g_2w^{(i+1,\infty)}_\gamma)^t\frac{\partial^2 V^{(i,\infty)}_\gamma}{\partial y^2}(g_1+g_2w^{(i+1,\infty)}_\gamma)]\\
	&\quad=
	{\nabla V^{(i,\infty)}_\gamma}^t(f+gu^{(i)}_\gamma+hw^{(i,\infty)}_\gamma)+
	\frac{1}{2}Tr[(g_1+g_2w^{(i,\infty)}_\gamma)^t\frac{\partial^2 V^{(i,\infty)}_\gamma}{\partial y^2}(g_1+g_2w^{(i,\infty)}_\gamma)]\\
	&\qquad+{\nabla V^{(i,\infty)}_\gamma}^tg(u^{(i+1)}_\gamma-u^{(i)}_\gamma)+\Big\langle h^t{\nabla V^{(i,\infty)}_\gamma}+Tr[g_2^t\frac{\partial^2 V^{(i,\infty)}_\gamma}{\partial y^2}g_1],(w^{(i+1,\infty)}_\gamma-w^{(i,\infty)}_\gamma)\Big\rangle\\
	&\qquad+\Big\langle\tfrac12Tr\big[g_2^t\frac{\partial^2V^{(i,\infty)}_\gamma}{\partial y^2}g_2\big]
			\big(w^{(i+1,\infty)}_\gamma+w^{(i,\infty)}_\gamma\big),\
			w^{(i+1,\infty)}_\gamma-w^{(i,\infty)}_\gamma\Big\rangle\\
	&\quad=-\Big(\ell(y)+\norm{u^{(i)}_\gamma}^2_R-\gamma^2\norm{w^{(i,\infty)}_\gamma}^2_P\Big)+\langle 2Rz, \cP_U(-z)-u^{(i)}_\gamma\rangle\\
	&\qquad+\Big\langle(2\gamma^2P-Tr[g_2^t\frac{\partial^2 V^{(i,\infty)}_\gamma}{\partial y^2}g_2])w^{(i,\infty)}_\gamma+\frac{1}{2}Tr[g_2^t\frac{\partial^2 V^{(i,\infty)}_\gamma}{\partial y^2}g_2](w^{(i+1,\infty)}_\gamma+w^{(i,\infty)}_\gamma),\\
	&\hspace{2cm} w^{(i+1,\infty)}_\gamma-w^{(i,\infty)}_\gamma\Big\rangle\\
	&=-\Big(\ell(y)+\norm{\cP_U(-z)}^2_R-\gamma^2\norm{w^{(i+1,\infty)}_\gamma}^2_P\Big)-\norm{u^{(i)}_\gamma-\cP_U(-z)}^2_R\\
	&\qquad+2(z+\cP_U(-z))^tR(\cP_U(-z)-u^{(i)}_\gamma)\\
	&\qquad-\big(w^{(i+1,\infty)}_\gamma-w^{(i,\infty)}_\gamma\big)^t
			\Big(\gamma^2P-\tfrac12Tr\big[g_2^t\frac{\partial^2V^{(i,\infty)}_\gamma}{\partial y^2}g_2\big]\Big)
			\big(w^{(i+1,\infty)}_\gamma-w^{(i,\infty)}_\gamma\big),
\end{align*}
where as earlier we denote $z=\frac{1}{2}R^{-1}g^t\nabla V^{(i,\infty)}_{\gamma}$, $g^t\nabla V^{(i,\infty)}_{\gamma}=2Rz$, and $u^{(i+1)}_\gamma=\cP_{U}(-z)$.
%
Finally, we arrive at
\begin{align*}
	&V^{(i+1,\infty)}_\gamma(x)-V^{(i,\infty)}_\gamma(x)\\
	&=-\E\int_0^\infty\Big(\big(w^{(i+1,\infty)}_\gamma-w^{(i,\infty)}_\gamma\big)^t
			\Big(\gamma^2P-\tfrac12Tr\big[g_2^t\frac{\partial^2V^{(i,\infty)}_\gamma}{\partial y^2}g_2\big]\Big)
			\big(w^{(i+1,\infty)}_\gamma-w^{(i,\infty)}_\gamma\big)
	\Big)dt\\
	&\qquad+\E\int_{0}^{\infty}\Big( 2(z+\cP_U(-z))^tR(\cP_U(-z)-u^{(i)}_\gamma) -\big\|u^{(i)}_\gamma-P_U(-z)\big\|_R^2\Big)\,dt\leq 0.
\end{align*}
Therefore, we get $V^{(i+1,\infty)}_\gamma(x)\leq V^{(i,\infty)}_\gamma(x)$.
Further $V^{(i,\infty)}_\gamma(x)\geq 0$ follows from Theorem \ref{lm2.1}.
\end{proof}
\begin{remark}
If $V^{(i,\infty)}_\gamma\in C^2(\Omega)$ satisfy \eqref{eq2.4}, $V_\gamma\in C^2(\Omega)\cap C^1(\bar\Omega)$ and $\{\nabla^m V^{(i,\infty)}_\gamma\}$ is equicontinuous, then we have $\{\nabla^m V^{(i,\infty)}_{\gamma}\}$ converges pointwise to $\nabla^m V_{\gamma}$ for $m=1,\,2$  in $\Omega$,   where $V_\gamma$ satisfies \eqref{eq2.3}.
\end{remark}
\begin{proof}
See proposition 4 in \cite{kk20}
for the proof in an analogous situation.
\end{proof}
\begin{lemma}\label{lm2.3}
If   $V_\gamma^{(i,\infty)}\in C^2(\R^d)$ for all $i\ge 0$ and  \eqref{eq2.9} holds,
then $u_\gamma^{(i+1)}$
is an asymptotically  stabilizing control on $\Omega$ for the undisturbed system ($w=0$), $d y=(f(y)+gu)\, dt+g_1(y)\,dW$.
\end{lemma}

\begin{proof}
From Theorem \ref{lm2.1}, $V^{(i,\infty)}_\gamma$ is positive definite.
Now along the trajectory $d y=(f(y)+gu^{(i+1)}_{\gamma})\, dt+g_1(y)\,dW$, the time derivative of $t\to V^{(i,\infty)}_{\gamma}(y)(t)$ can be obtained as
\begin{align*}
	&{\nabla V^{(i,\infty)}_\gamma}^t
	(f+gu^{(i+1)}_\gamma)+\frac{1}{2}Tr[(g_1)^t\frac{\partial^2 V^{(i,\infty)}_\gamma}{\partial y^2}(g_1)]\\
	&\quad=
	{\nabla V^{(i,\infty)}_\gamma}^t(f+gu^{(i)}_\gamma)+
	\frac{1}{2}Tr[(g_1)^t\frac{\partial^2 V^{(i,\infty)}_\gamma}{\partial y^2}(g_1)]+{\nabla V^{(i,\infty)}_\gamma}^tg(u^{(i+1)}_\gamma-u^{(i)}_\gamma).
\end{align*}
From \eqref{eq2.9}, it follows that
\begin{align*}
	&\nabla V^{(i,\infty)}_\gamma(y)^t(f(y)+gu^{(i)}_\gamma)+\frac{1}{2}Tr[g_1^t\frac{\partial^2 V^{(i,\infty)}_\gamma}{\partial y^2}g_1]\\=-& \frac{1}{2}\Big(\nabla V^{(i,\infty)}_\gamma(y)^th+Tr[g_1^t\frac{\partial^2 V^{(i,\infty)}_\gamma}{\partial y^2}g_2]\Big)(2\gamma^2 P-Tr[g_2^t\frac{\partial^2 V^{(i,\infty)}_\gamma}{\partial y^2}g_2])^{-1}\Big(h^t\nabla V^{(i,\infty)}_\gamma(y)\\
	&\qquad+Tr[g_2^t\frac{\partial^2 V^{(i,\infty)}_\gamma}{\partial y^2}g_1]\Big)-\norm{u^{(i)}_\gamma}^2_R-\ell(y).
\end{align*}
Also ${\nabla V^{(i,\infty)}_\gamma}^tg(u^{(i+1)}_\gamma-u^{(i)}_\gamma)$ is estimated in Theorem \ref{lm2.2}.
Hence,  we get
\begin{align*}
	&{\nabla V^{(i,\infty)}_\gamma}^t
	(f+gu^{(i+1)}_\gamma)+\frac{1}{2}Tr[(g_1)^t\frac{\partial^2 V^{(i,\infty)}_\gamma}{\partial y^2}(g_1)]\\
	&\quad=-\ell(y)-\norm{\cP_U(-z)}^2_R-\norm{u^{(i)}_\gamma-\cP_U(-z)}^2_R+2(z+\cP_U(-z))^tR(\cP_U(-z)-u^{(i)}_\gamma)\\
	&\qquad-\frac{1}{2}\Big(\nabla V^{(i,\infty)}_\gamma(y)^th+Tr[g_1^t\frac{\partial^2 V^{(i,\infty)}_\gamma}{\partial y^2}g_2]\Big)(2\gamma^2 P-Tr[g_2^t\frac{\partial^2 V^{(i,\infty)}_\gamma}{\partial y^2}g_2])^{-1}\Big(h^t\nabla V^{(i,\infty)}_\gamma(y)\\
	&\qquad+Tr[g_2^t\frac{\partial^2 V^{(i,\infty)}_\gamma}{\partial y^2}g_1]\Big)< 0,
\end{align*}
and the asymptotic stabilizing property of $u^{(i+1)}_\gamma$ follows.
\end{proof}

\section{Numerical examples}
In this section we solve the HJI equation by policy iteration Algorithm \ref{alg:sg1} for the deterministic, and by Algorithm \ref{alg:sg2} for the stochastic case. An implicit upwind scheme is used for  solving backward PDEs, and a bisection algorithm is employed for finding $\gamma^*$. For more details on the upwind scheme, see \cite{Aal17} and \cite{kk20}.
To address the difficulty arising from  the need to  initialize Algorithm \ref{alg:sg1} with an asymptotically stabilizing solution also  in  the case  when the origin is not asymptotically stable,  a discount factor $\lambda>0$  is introduced in the cost. With this change the algorithm always converged satisfactorily. The result thus obtained can be used to initialize the algorithm without discount factor. An alternative possibility for initialization is provided by the choice given by the Riccati feedback law, which is obtained from considering the linearized system. For each of the numerical tests that we shall describe below, the stability properties of the steady state and its consequences for the initialization will be addressed.

For completeness,  we briefly discuss the upwind scheme that is used, where for simplicity, we focus on the 1D deterministic model. For a fixed $\gamma$ we describe the scheme to solve the HJI equation \eqref{eq:aux1}. Let $(u,V)$ be the initial  pair, suppressing the dependence on $\gamma$.
Let $n$ stand for the iteration/time loop, let $i$ denote the  numbering of the mesh points $x_i$, and  let $dt$ stand for the time step. Set $V_i=V(x_i)$, $u_i=u(x_i)$ and $w_i=w(x_i)$ with $i=1,\ldots, I$, where $I$ is the total number of mesh points.
We use the forward difference operator $\nabla V_{i,F}=\frac{V_{i+1}-V_i}{\Delta x}$, $w_{i,F}=\frac{P^{-1}}{2\gamma^2}\nabla V_{i,F}$ whenever the drift  $S_{i,F}=f(x_i)+g(x_i)u_i+h(x_i)w_{i,F}>0$, and the backward difference operator $\nabla V_{i,B}=\frac{V_{i}-V_{i-1}}{\Delta x}$,  if the drift $S_{i,B}=f(x_i)+g(x_i)u_i+h(x_i)w_{i,B}<0$, and $ D\bar V_i=2P\gamma^2\bar w_i$ with $\bar w_i=-f(x_i)/h(x_i)$ for $S_{i,F}\leq 0\leq S_{i,B}$, where $\Delta x=x_{i+1}-x_{i}$ is the mesh size.
Then an approximation of the derivative of the value function can be constructed as
$$\nabla V^n_{iupwind}=\nabla V^n_{i,F}\one_{S_{i,F}>0}+\nabla V^n_{i,B}\one_{S_{i,B}<0}+ D\bar V^n_i\one_{S_{i,F}\leq 0\leq S_{i,B}},$$ where $\one$ is the characteristic function.
Finally, taking $\ell(x)=x^2$, we solve
\begin{align}\label{eq3.3}
-\frac{(V^{n}_i-V^{(n-1)}_i)}{dt}&-\lambda V^n_i+\nabla V^n_{iupwind}(f(x_i)+gu^n_i+hw^n_i)\notag\\
&+x_i^2+\norm{u^n_i}^2_R-\gamma^2\norm{w^n_i}^2_{P}=0.
\end{align}
Further, we get the  disturbance update
$w^{n+1}_i=\frac{1}{2\gamma^2}P^{-1}h(x_i)^t\nabla V^n_{iupwind}$. Once the disturbance loop is completed by satisfying the tolerance, we update the control similarly by considering the upwind derivative from the value function obtained after completing the disturbance loop.
Subsequently,  using a bisection algorithm we can get the smallest value of $\gamma$, and the corresponding value of  $u_{\gamma^*}$, $w_{\gamma^*}$ and $V_{\gamma^*}$.
Similarly, for the stochastic case, the additional second order term is discretized by  $D^2V^n_i=(V^n_{i+1}-2V^n_i+V^n_{i-1})/\Delta x^2$, and an  analogous upwind scheme is used.

\subsection{Test 1: 1D equation}
In the first example we test a linear unstable system.  We first employ as discount factor $\lambda=0.005$. Then the upwind solutions for $\lambda=0.005$  and that  for $\lambda=0$ are compared, where in both cases we take the Riccati solution as the initialization. We also report on  the difference  between the upwind and the Riccati solutions.

We consider
$$\underset{u(\cdot)\in \cU}{\min} \underset{w(\cdot)\in \cW}{\max}\; \cJ(u(\cdot),w(\cdot);(x_0))=\int_{0}^{\infty}e^{-\lambda t}\Big(\norm{x}^2+\norm{u}^2_{R}-\gamma^2\norm{w}^2_P\Big) dt,$$
subject to the  one dimensional  equation
\begin{align}\label{e4.1}
\frac{dx}{dt}=0.5x+u+w, \quad \quad x(0)=x_0.
\end{align}
Further we choose  $R=0.1$, $P=1$, $dt=0.005$, $\Omega=(-2,2)$ and $\gamma=1$.
\begin{table}[ht]
\centering
\caption{$L^2$ and $L^\infty$ errors for the upwind scheme in Test 1 with $\lambda=0$.}
\begin{tabular}{c c c c c}
	\hline
	$Mesh points$ & $\norm{V-V_{upwind}}_{L^2}$ & Conv. order & $\norm{V-V_{upwind}}_{L^\infty}$ & Conv. order
	\\  \hline \hline
	$200$ & $0.07735$            &      &  $0.009425$       &    \\ \hline
	$400$ & $0.017396$             & $2.15$ &  $0.0015028$       & $2.64$   \\ \hline
	$600$ & $0.0011153$            & $6.77$ & $7.87\times 10^{-5}$      &  $7.27$    \\ \hline
	$800$ & $9.65\times 10^{-4}$            & $0.503$ & $5.9\times 10^{-5}$     & $1.002$   \\ \hline
	$1000$ & $8.63 \times 10^{-4}$ & $0.502$ &  $4.72\times 10^{-5}$      &  $1.001$  \\ \hline
	$1200$ & $7.87\times 10^{-4}$ & $0.502$&  $3.93\times 10^{-5}$       &  $1.001$   \\ \hline
	$1400$ & $7.29\times 10^{-4}$ &$0.502$ &$3.37\times 10^{-5}$&    $1.001$  \\ \hline
	\label{table:1}
\end{tabular}
\end{table}
Let  $e_\infty:=\norm{V-V_{upwind}}_{L^\infty}$ and   $e_2:=\norm{V-V_{upwind}}_{L^2}$
denote the  $L^\infty$ and $L^2$ norm errors of the  value function,  where $V$ corresponds to the Riccati solution. We calculate the
order of convergence $\alpha:=
ln(e_{new}/e_{old})/ln(h_{new}/h_{old})$, where $e_{new}$ and $e_{old}$ correspond to the current and previous steps, respectively, and $h_{new}$, $h_{old}$ are the current and previous step mesh sizes respectively. Since $h$ is inversely proportional to the number of mesh points $I$, the  convergence order can be written as $\alpha=ln(e_{new}/e_{old})/ln(I_{old}/I_{new})$.	
%
Next we solve with $\lambda=0$ in the value function,  and compare the corresponding upwind solution with the  $\lambda=0.005$ case.
From Table \ref{table:1} we see that the $L^2$ and $L^\infty$ order of convergence for $V-V_{upwind}$ are $0.5$ and $1$, respectively. Moreover,  we observe that the difference between the value function for $\lambda=0$ and $\lambda=0.005$  is
$\norm{V_0-V_{0.005}}_{L^2}=5.93\times 10^{-4}$ and $\norm{V_0-V_{0.005}}_{L^\infty}=3.82\times 10^{-5}$, where $V_\lambda$ denotes the upwind value function with the corresponding $\lambda$. 	
\subsection{Test 2: 2D nonlinear equation}

We consider next the following minimization problem
$$\underset{u(\cdot)\in \cU}{\min} \underset{w(\cdot)\in \cW}{\max}\; \cJ(u(\cdot),w(\cdot);(x_0,y_0))=\int_{0}^{\infty}e^{-\lambda t}\Big(\norm{x}^2+\norm{y}^2+\norm{u}^2_{R}-\gamma^2\norm{w}^2_P\Big) dt,$$
associated to the following two-dimensional nonlinear equation
\begin{align}\label{eqx2.1}
\frac{dx}{dt}&=-x-y,\quad x(0)=x_0,\notag\\
\frac{dy}{dt}&=-(1-x^2)y+u+w, \quad y(0)=y_0,
\end{align}
where $R=0.01$, $P=1$, $\gamma=1$, $\lambda=0.005$ and $dt=0.001$.
Let us observe that the origin is a locally stable equilibrium for this system, but it is not globally stable.  The HJI equation was  solved for $V_{0.005}$ in $\Omega=(-2,2)^2$  with
$100\times 100$ equidistant mesh points. Observe that this choice of $\Omega$ is not contained in the neighborhood of asymptotic stability for the uncontrolled system \eqref{eqx2.1}.
Next the solution thus obtained was used as initialization to solve with zero discount factor. The associated value function is denoted by $V_{0}$. The $L^\infty$ and $L^2$ errors were found to be  $\norm{V_{0}-V_{0.005}}_{L^\infty}=6.35\times 10^{-5}$ and $\norm{V_{0}-V_{0.005}}_{L^2}=0.0017$.

\subsection{Test 3: Lorenz system}
Here we treat the controlled Lorenz system, with both, stable and unstable parameter settings. We consider
\begin{align}\label{3dcost}
\underset{u(\cdot)\in \cU}{\min}\;\underset{w(\cdot)\in \cW}{\max}\; \cJ(u(\cdot),w(\cdot);(x_0,y_0,z_0))=\int_{0}^{\infty}e^{-\lambda t}\Big(\norm{x}^2+\norm{y}^2+\norm{z}^2+\norm{u}^2_{R}-\gamma^2\norm{w}^2_P\Big) dt,
\end{align}
subject to the three-dimensional Lorenz system, with  control and disturbance appearing in  the second equation,
\begin{align}
\frac{dx}{dt}&=\sigma(y-x)\label{eq3.4}, \quad x(0)=x_0,\\
\frac{dy}{dt}&=x(\rho-z)-y+u+w\label{eq3.5},\quad -1\leq u\leq 1,\quad y(0)=y_0,\\
\frac{dz}{dt}&=xy-\beta z\label{eq3.6}, \quad z(0)=z_0,
\end{align}
where  $\sigma>1$, $\rho>0$ and $\beta>0$. The system \eqref {eq3.4}-\eqref{eq3.6} has 3 steady state solutions,  $C^0=(0,0,0)$, $C^+=\Big(\sqrt{\beta(\rho-1)}, \sqrt{\beta(\rho-1)},\rho-1\Big)$ and $C^-=\Big(-\sqrt{\beta(\rho-1)}, -\sqrt{\beta(\rho-1)},\rho-1\Big)$. If $\rho<1$, all the steady state solutions are stable.

We first consider such a stable unconstrained situation and choose
$\rho=0.5$, $\sigma=10$, $\beta=8/3$. The other parameter settings are chosen to be $R=0.01$, $P=1$, $\gamma=2$, $dt=0.01$, and  $40\times 40 \times 40$ mesh points in the domain $\Omega=(-2,2)^3$.
We calculate the value function for $\lambda=0.005$ and denote it by $V_{0.005}$. Moreover, using $V_{0.005}$ as initialization, we also obtain the solution for $\lambda=0$ denoted by $V_0$. The $L^\infty$ error between these two solutions is found to be $\norm{V_0-V_{0.005}}_{L^\infty}=0.0042$.

We carried out many more experiments for the case of an unstable system with parameter settings $\sigma=10$, $\beta=8/3$, and $\rho=2$.
Again the HJI equation was solved over  $\Omega=(-2,2)^3$ with $40\times 40 \times 40$ spatial discretized points,  with $R=0.1$, $P=1$,  $dt=2$. To avoid difficulties with initialization we took $\lambda=0.05$ and chose the Riccati solution as the initial guess for policy iteration. 
We find that  $\gamma^*=1.9196$ and $\gamma^*=7.1136$ for the unconstrained and constrained cases respectively, with associated costs  $\cJ=1.86$ and $\cJ=3.89$. When discussing the behavior of trajectories of system \eqref{eq3.4}-\eqref{eq3.6} below,  we do this for $[x_0,y_0,z_0]=[-1.5,-1.5,1.5]$ as initial conditions.
\begin{figure}[!h]
\centering
(i)	\includegraphics[width=0.45\textwidth]{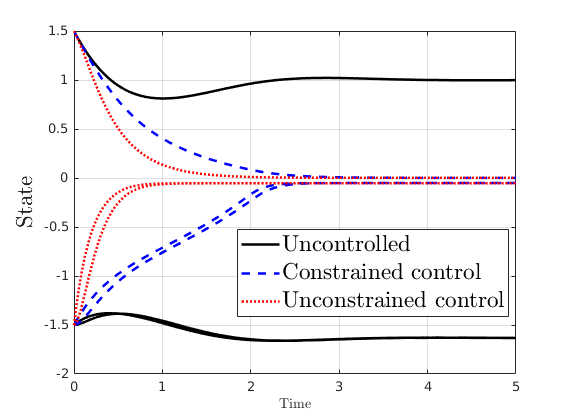}
(ii)\includegraphics[width=0.45\textwidth]{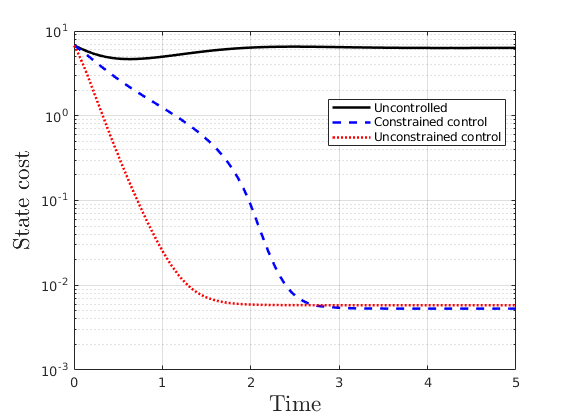}
(iii)	\includegraphics[width=0.45\textwidth]{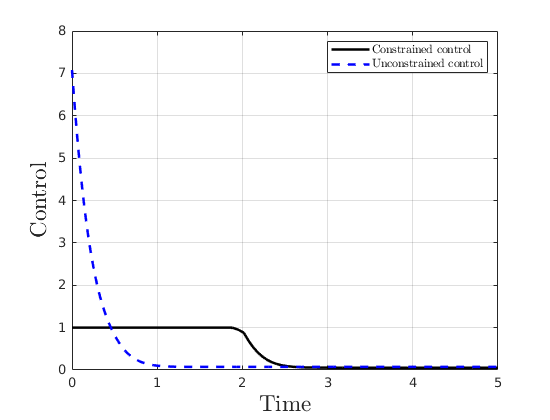}
(iv)	\includegraphics[width=0.45\textwidth]{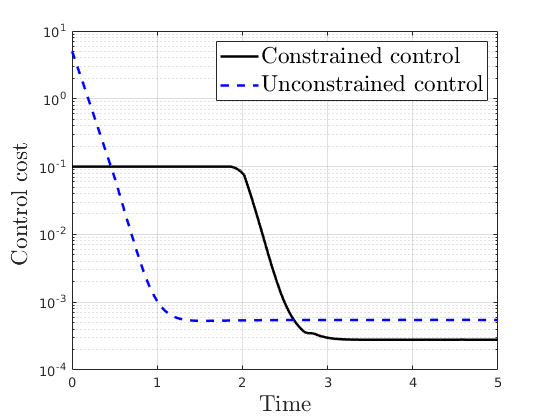}
(v)	\includegraphics[width=0.45\textwidth]{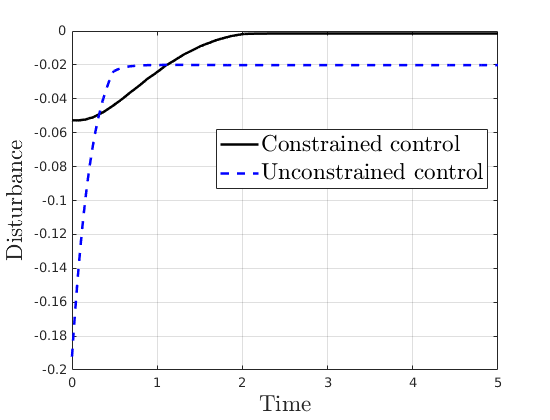}
(vi)\includegraphics[width=0.45\textwidth]{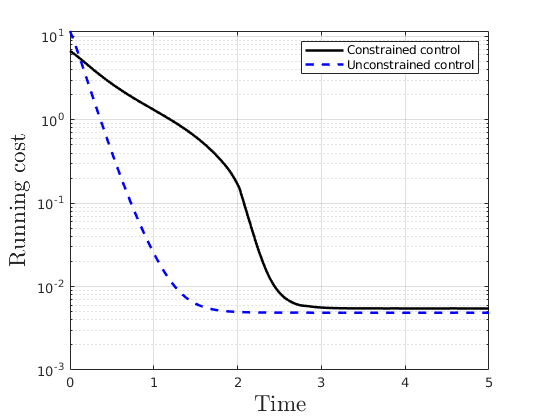}
\caption{Test 3: Deterministic case.  {\bf i)} State,  {\bf ii)} State cost $\norm{x(t)}^2+\norm{y(t)}^2+\norm{z(t)}^2$, {\bf iii)} Control, {\bf iv)} Control cost $\norm{u(t)}^2_R$, {\bf v)} Disturbance, {\bf vi)} Running cost $\norm{x(t)}^2+\norm{y(t)}^2+\norm{z(t)}^2+\norm{u(t)}^2_R-{\gamma^*}^2\norm{w(t)}^2_P$.}\label{fig:ex2.1}
\end{figure}

We  next depict results for the unstable parameter choice settings: From Figure \ref{fig:ex2.1}(i)-(ii) we observe  that the uncontrolled state  and the corresponding state cost  do not settle towards zero.
The states of the controlled system, both  for the unconstrained and the constrained HJI controls, documented in Figure  \ref{fig:ex2.1}(iii), approach zero. As expected this is faster for  unconstrained control case  than for the constrained case.  Similar behavior for the control cost  and running cost are shown in Figure \ref{fig:ex2.1}(iv) and Figure \ref{fig:ex2.1}(vi) respectively. In Figure \ref{fig:ex2.1}(v) the disturbance for the constrained control tends to zero faster compared to the disturbance for the unconstrained case since the constrained control cannot handle as much disturbance as the unconstrained one.

\begin{figure}[!h]
\centering
(i)\includegraphics[width=0.45\textwidth]{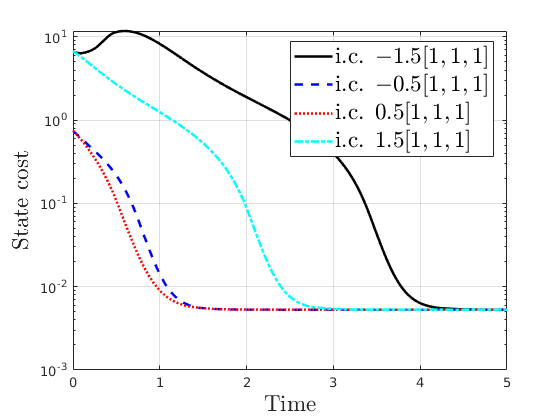}
(ii)	\includegraphics[width=0.45\textwidth]{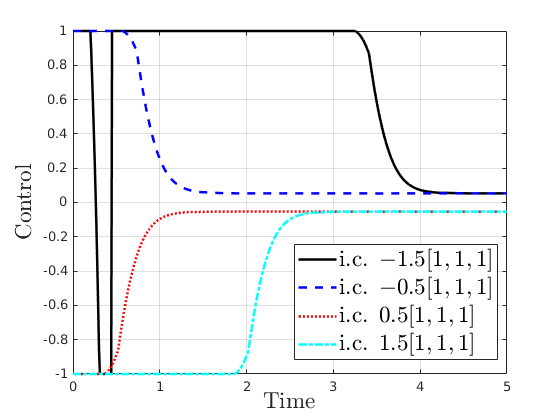}
\caption{Test 3: Deterministic case. State cost and Control trajectory with different initial conditions.  {\bf i)} State cost {\bf ii)} Control. }\label{fig:x}
\end{figure}
We show the behavior of the state cost $\norm{x(t)}^2+\norm{y(t)}^2+\norm{z(t)}^2$ and the control trajectory with respect to different initial conditions in Figure \ref{fig:x} (i)-(ii).	

%
\begin{figure}[!h]
\centering
(i)\includegraphics[width=0.45\textwidth]{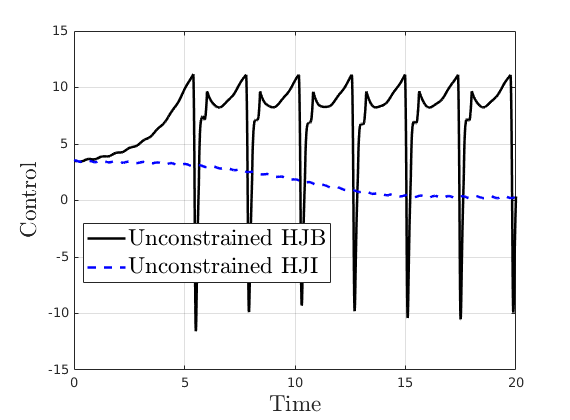}
(ii)	\includegraphics[width=0.45\textwidth]{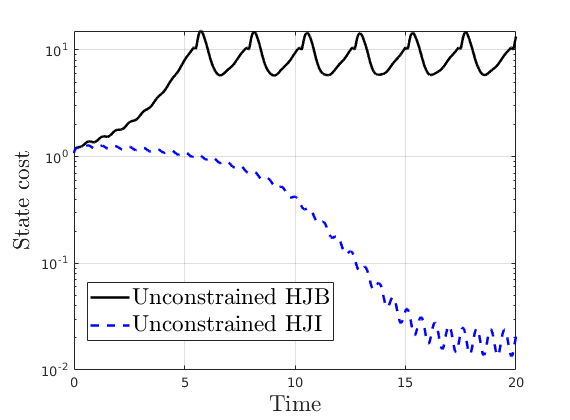}
\caption{Test 3: Unconstrained case, Comparison between the HJI and HJB synthesis with noise $w_1=-0.6u_{HJB}+0.1sin(10t)$, $[x_0,y_0,z_0]=[-0.6,-0.6,-0.6]$.   {\bf i)} Control, {\bf ii)}  State cost. }\label{fig:ex2.2}
\end{figure}
\begin{figure}[!h]
\centering
(i)\includegraphics[width=0.45\textwidth]{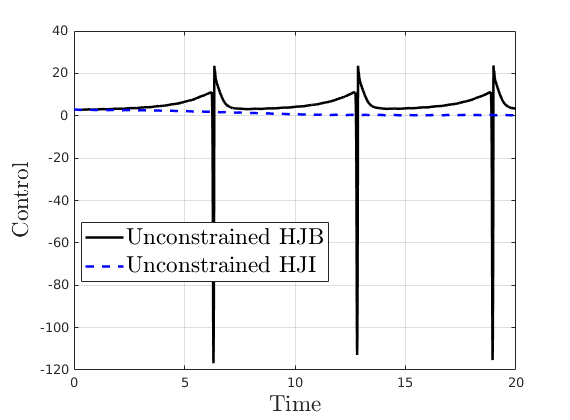}
(ii)	\includegraphics[width=0.45\textwidth]{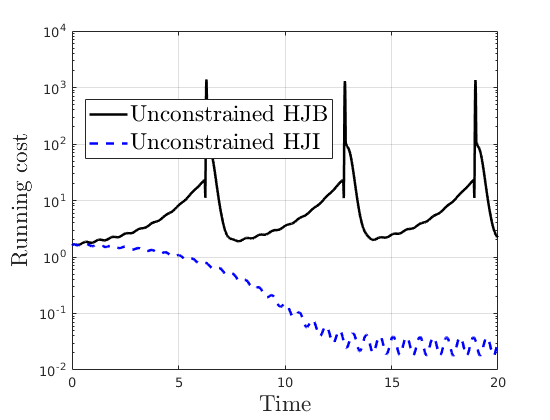}
\caption{Test 3: Unconstrained case, Comparison between the HJI and HJB synthesis with noise $w_2=-0.6u_{HJI}+0.1sin(10t)$, $[x_0,y_0,z_0]=[-0.5,-0.5,-0.5]$.   {\bf i)} Control, {\bf ii)}  Running cost. }\label{fig:ex2.3}
\end{figure}

We next show comparisons between the HJI and HJB synthesis in presence of more adverse noise and different initial conditions. If the HJB control is used within a disturbance of the form $w_1=-0.6u_{HJB}+0.1sin(10t)$ then the unconstrained trajectory for both  HJB control and state cost oscillate, whereas the unconstrained HJI control and state cost are successful in noise rejection, and settle at zero, see Figure \ref{fig:ex2.2}(i)-(ii). We mention that if we choose the disturbance of the form $w_1=-0.5u_{HJB}+0.1sin(10t)$, then there is no spike in the HJB case and the corresponding state and control trajectories converge towards zero.

Similarly, if the noise is of the form $w_2=-0.6u_{HJI}+0.1sin(10t)$, in Figure  \ref{fig:ex2.3}(i) we observe that the unconstrained HJI control is successful in  handling the noise, whereas the unconstrained HJB control fails and starts to oscillate. Similar behavior is shown for the unconstrained running cost in Figure \ref{fig:ex2.3}(ii).

\begin{figure}[!h]
\centering
(i)\includegraphics[width=0.45\textwidth]{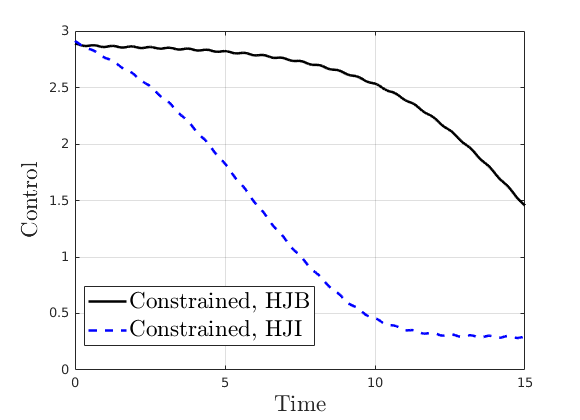}
(ii)	\includegraphics[width=0.45\textwidth]{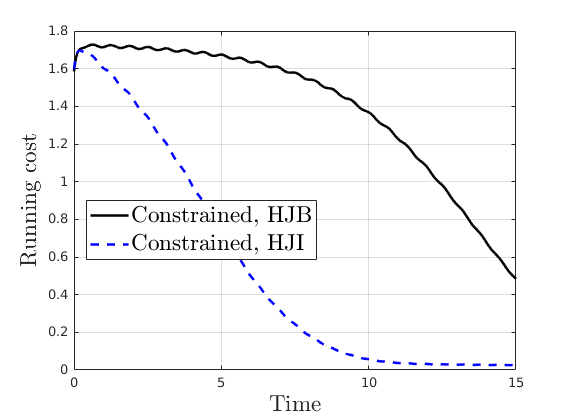}
\caption{Test 3: Constrained case, comparison between the HJI and HJB synthesis with noise $w_1=-0.6u_{HJB}+0.01sin(10t)$, $-4\leq u\leq 4$, $[x_0,y_0,z_0]=[-0.5,-0.5,-0.5]$, other parameters as before.   {\bf i)} Control, {\bf ii)}  Running cost. }\label{fig:ex2.4}
\end{figure}
\begin{figure}[!h]
\centering
(i)\includegraphics[width=0.45\textwidth]{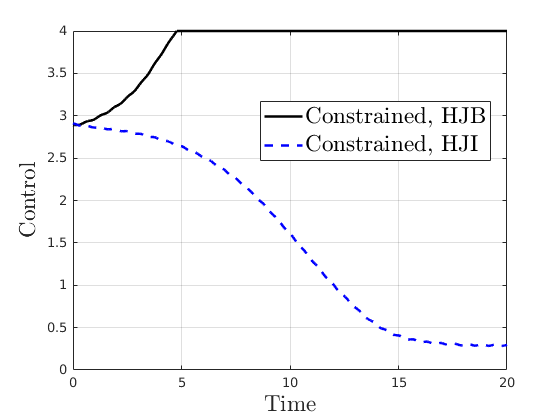}
(ii)	\includegraphics[width=0.45\textwidth]{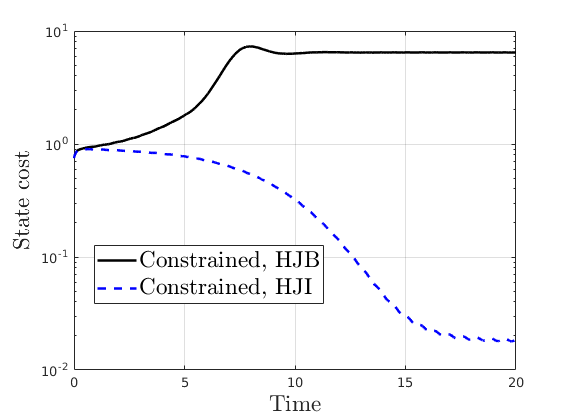}
\caption{Test 3: Constrained case, comparison between the HJI and HJB synthesis with noise $w_2=-0.6u_{HJI}+0.01sin(10t)$, $-4\leq u\leq 4$, $[x_0,y_0,z_0]=[-0.5,-0.5,-0.5]$,  other parameters as before.   {\bf i)} Control, {\bf ii)}  State cost. }\label{fig:ex2.5}
\end{figure}

To compare the constrained HJI and HJB synthesis in the presence of noise $w_1$ and $w_2$, we only modified the constraint in \eqref{eq3.5} to be $-4\leq u\leq 4$. With disturbances of the form $w_1=-0.6u_{HJB}+0.01sin(10t)$,  the HJI running cost tends to zero faster compared to the HJB cost, as seen in Figure \ref{fig:ex2.4}.
Further, with noise of the form $w_2=-0.6u_{HJI}+0.01sin(10t)$, the HJI synthesis is successful in driving the control and state cost to their equilibrium position, whereas the HJB synthesis fails as observed from Figure \ref{fig:ex2.5}.

The study of stochastic zero-sum differential games is motivated by the need to design robust controls for chaotic systems such as the Lorenz system, where sensitivity to noise and modeling errors severely limits classical approaches. Reformulating the problem as a two-player stochastic game, with the controller minimizing and the disturbance maximizing, naturally leads to a second-order HJI equation.
For the stochastic case, we put the additive noise $g_1=0.5$ in the third equation \eqref{eq3.6}. The other parameters and the initial condition are the same as in the deterministic Lorenz system. We obtain $\gamma^*=1.9135$ and $\gamma^*=7.1136$ for the stochastic unconstrained, respectively constrained HJI cases.
\begin{figure}[!h]
\centering
(i)	\includegraphics[width=0.45\textwidth]{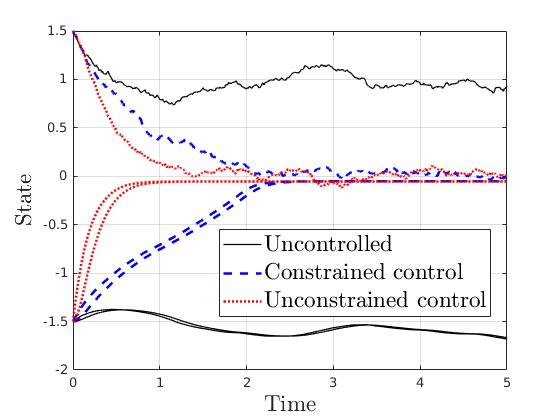}
(ii)\includegraphics[width=0.45\textwidth]{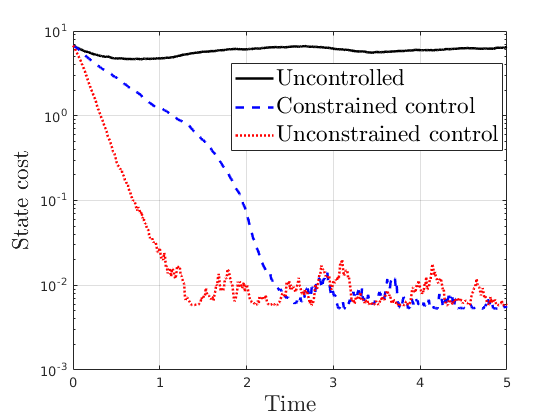}
(iii)	\includegraphics[width=0.45\textwidth]{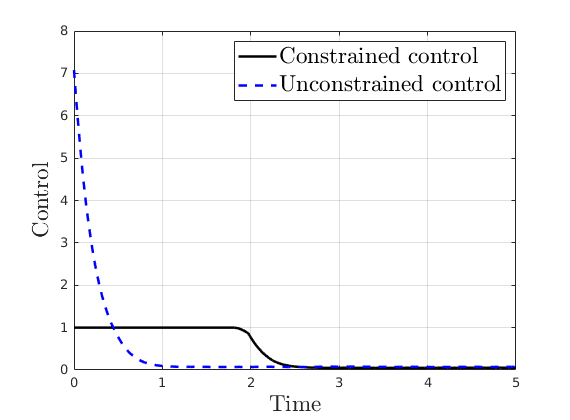}
(iv)	\includegraphics[width=0.45\textwidth]{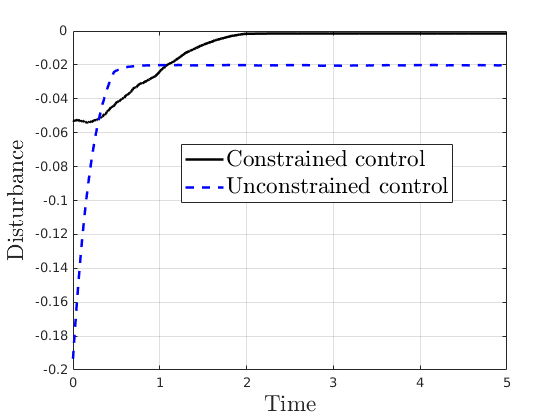}
\caption{Test 3: Stochastic case, $[x_0,y_0,z_0]=[-1.5,-1.5, 1.5]$.  {\bf i)} State,  {\bf ii)} State cost $\norm{x(t)}^2+\norm{y(t)}^2+\norm{z(t)}^2$, {\bf iii)} Control,  {\bf iv)} Disturbance.}\label{fig:ex2.6}
\end{figure}
From Figure \ref{fig:ex2.6} we observe that the behavior of the stochastic Lorenz  system is similar to the deterministic case except for the appearance of the randomness of the optimal state.

\begin{figure}[!h]
\centering
(i)\includegraphics[width=0.45\textwidth]{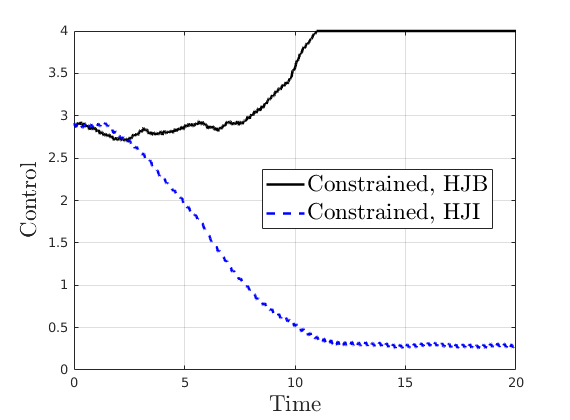}
(ii)	\includegraphics[width=0.45\textwidth]{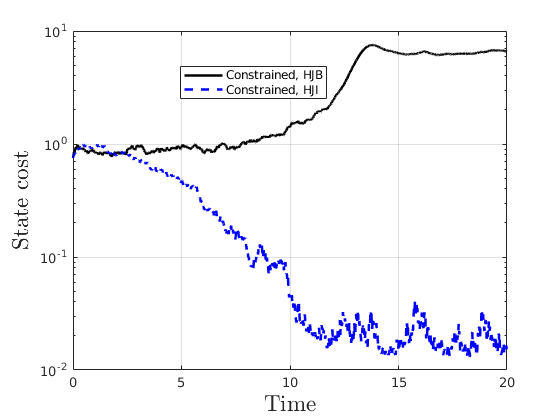}
\caption{Test 3: Constrained case, comparison between the HJI and HJB synthesis with noise $w_1=-0.6u_{HJB}+0.01sin(10t)$, $-4\leq u\leq 4$, $[x_0,y_0,z_0]=[-0.5,-0.5,-0.5]$, other parameters as before.   {\bf i)} Control, {\bf ii)}  State cost. }\label{fig:ex2.8}
\end{figure}

With the disturbance of the form $w_1=-0.6u_{HJB}+0.01sin(10t)$, comparisons between the   HJB and HJI synthesis under stochastic dynamics are shown in Figure \ref{fig:ex2.8}(i)-(ii) for the constrained case. The reader will notice that the HJI control and state cost both tend to settle to zero, while this does not happen in the HJB case. This difference occurs if the constraints are lifted.
%
%
%


We also carried out tests as $\lambda \to 0$. In the unconstrained case, with the solution to $V_{0.05}$ as an initialization,  $dt =0.01$ and keeping the  other parameters as before, we obtain for the  $L^\infty$ error  $\norm{V_0-V_{0.05}}_{L^\infty(\Omega)}=0.08$. If $dt$ is reduced to $dt=0.001$, then  $\norm{V_0-V_{0.05}}_{L^\infty(\Omega)}=0.03$. Further, for $\lambda=0.005$, we obtained $\norm{V_0-V_{0.005}}_{L^\infty(\Omega)}=3.1\times 10^{-5}$. In the constrained case, with $-4\leq u\leq 4$, then keeping $dt=0.01$ and $\gamma=10$, we obtain $\norm{V_0-V_{0.005}}_{L^\infty(\Omega)}=0.02$ and $\norm{V_0-V_{0.05}}_{L^\infty(\Omega)}=0.12$.

\subsection{Test 4: Van der Pol equation}
Here we consider the controlled Van der Pol equation. It is well-known that the unforced Van der Pol equation has a stable limit cycle, the origin is unstable. While admittedly  we cannot find a stabilizing initial condition, it appears to be of interest to test the proposed methodology for this challenging example. We consider
$$\underset{u(\cdot)\in \cU}{\min} \underset{w(\cdot)\in \cW}{\max}\; \cJ(u(\cdot),w(\cdot);(x_0,y_0))=\int_{0}^{\infty}e^{-\lambda t}\Big(\norm{x}^2+\norm{y}^2+\norm{u}^2_{R}-\gamma^2\norm{w}^2_P\Big) dt,$$
subject to the following two-dimensional Van der Pol equation
\begin{align}\label{eq4.1}
	\frac{dx}{dt}&=y,\quad x(0)=x_0,\notag\\
	\frac{dy}{dt}&=(1-x^2)y-x+u+w, \quad -1\leq u\leq 1,\quad y(0)=y_0,
\end{align}
where $R=0.01$, $P=1$ and $[x_0,y_0]=[1, -1]$.
We solve the HJI and HJB equations over $\Omega=(-2,2)^2$  choosing the time step $dt=2$ for the artificial term in the upwind scheme along with  $60\times 60$ spatial discretization points. Solving the HJI equation  with $\lambda=0.05$ and the Riccati solution as an initialization, we obtain the lowest value of $\gamma$ as $\gamma^*=7.9$, respectively $\gamma^*=0.65$, for the constrained  and unconstrained control cases. The total cost becomes $\cJ=4.69$ for the constrained control case, and $\cJ=1.04$ for the unconstrained case.

\begin{figure}[!h]
	\centering
	(i)	\includegraphics[width=0.45\textwidth]{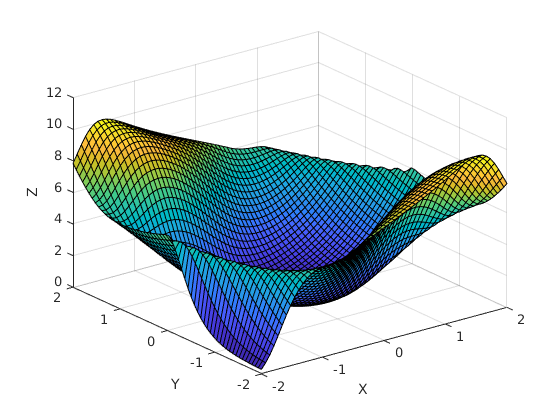}
	\caption{Test 4: Deterministic case.  {\bf i)} Value function, constrained case. }\label{fig:x2}
\end{figure}
The plot of the value function for the constrained case is documented in Figure \ref{fig:x2}.  It suggests that  the value function is not differentiable near boundary values of the domain, but that  it is smooth in a neighborhood of the origin.

\begin{figure}[!h]
	\centering
	(i)	\includegraphics[width=0.45\textwidth]{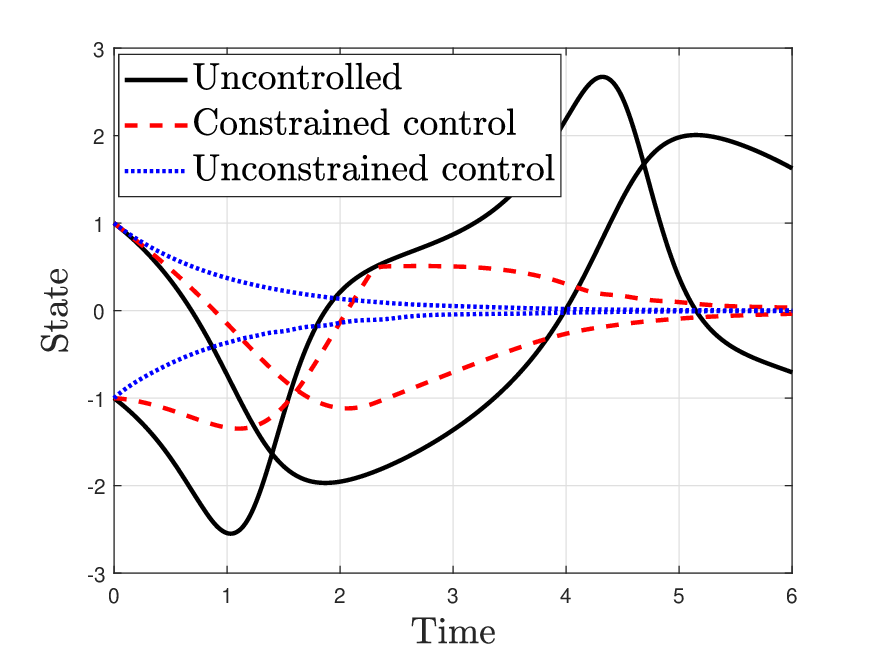}
	(ii)\includegraphics[width=0.45\textwidth]{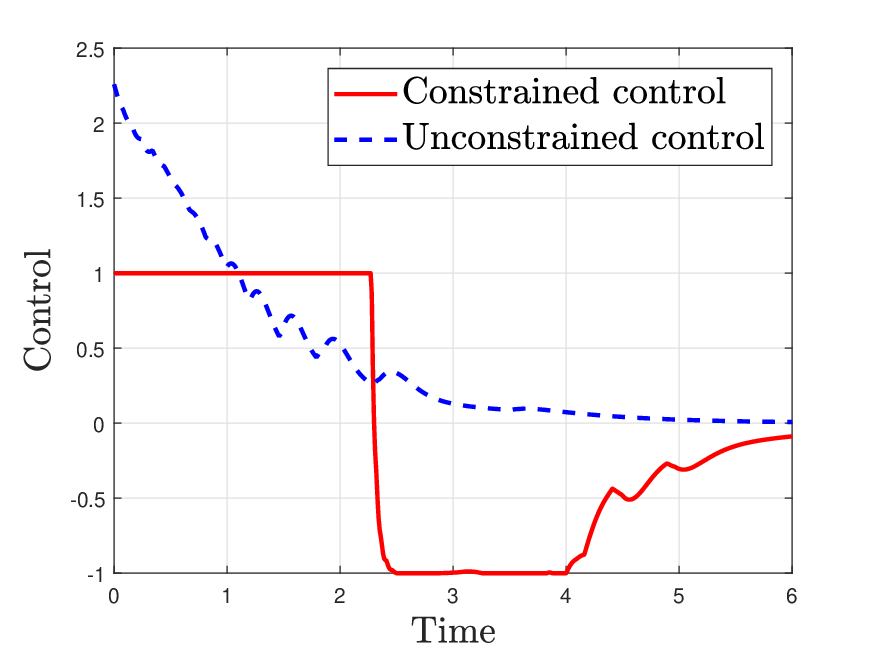}
	(iii)	\includegraphics[width=0.45\textwidth]{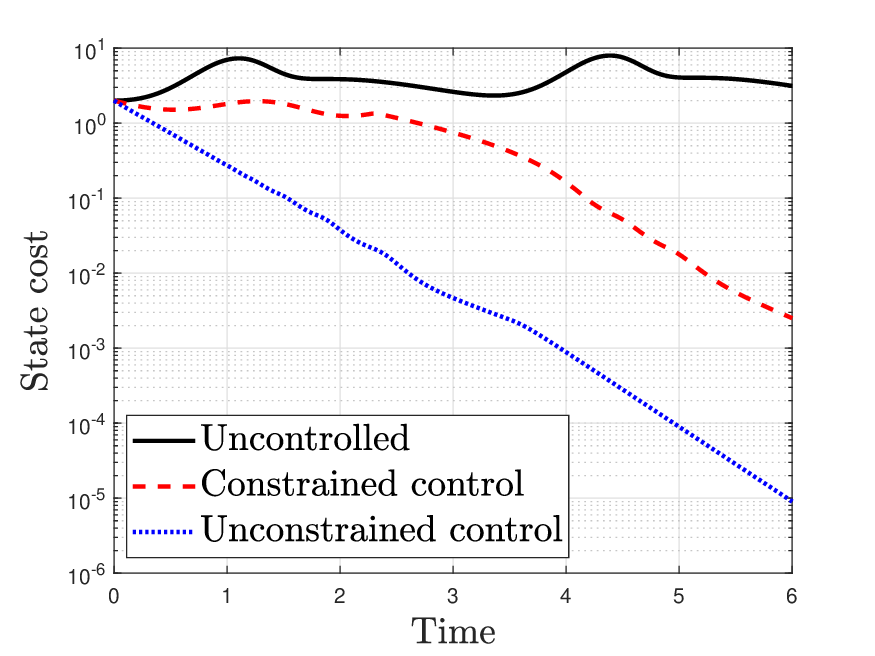}
	(iv)	\includegraphics[width=0.45\textwidth]{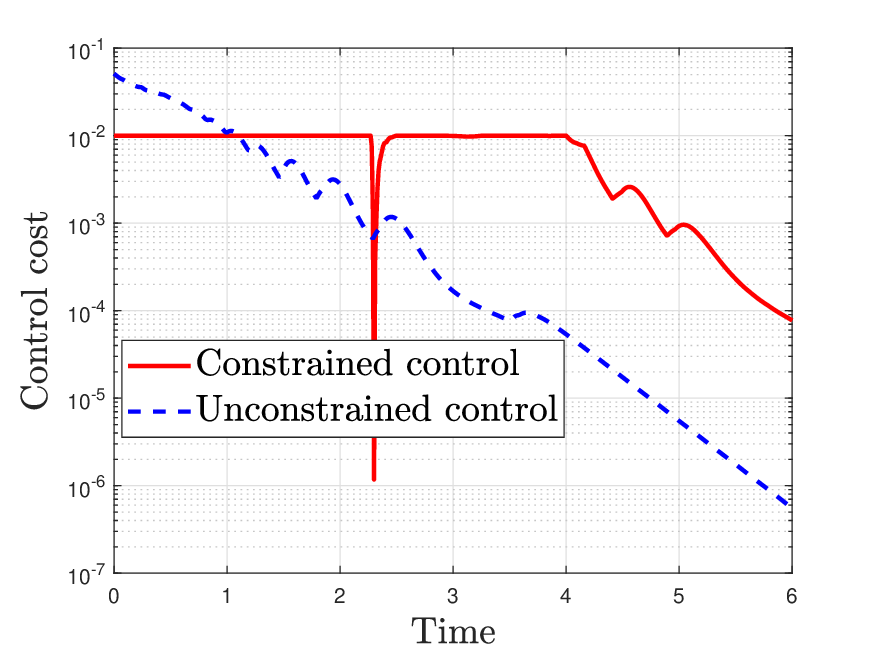}
	(v)	\includegraphics[width=0.45\textwidth]{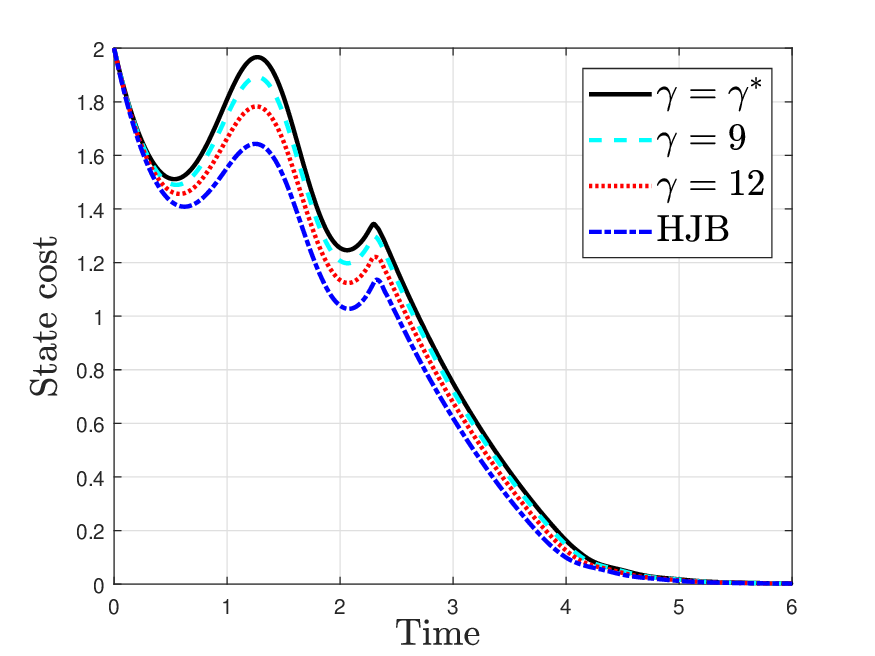}
	(vi)\includegraphics[width=0.45\textwidth]{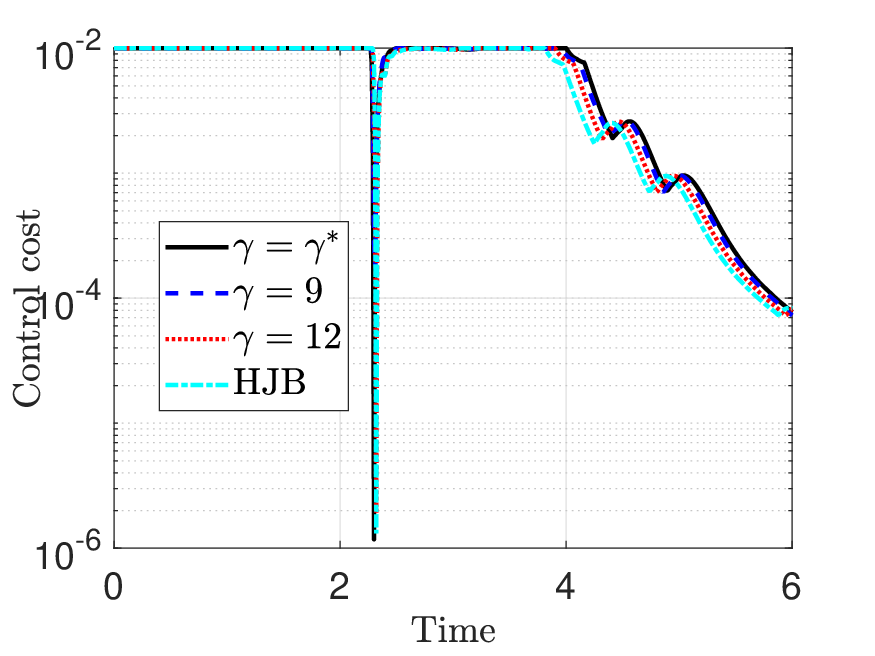}
	\caption{Test 4: Deterministic case.  {\bf i)} State,  {\bf ii)} Control, {\bf iii)} State cost, {\bf iv)} Control cost, {\bf v)} State cost convergence of the HJI to HJB synthesis in the constrained case, {\bf vi)} Control cost convergence of the HJI to HJB synthesis in the constrained case. }\label{fig:ex1.1}
\end{figure}

From Figure \ref{fig:ex1.1}(i) we  observe that the uncontrolled solution $(u=0=w)$ of  \eqref{eq4.1} does not settle at zero, but applying either constrained or unconstrained control, the controlled solutions converge to zero. The decay rate is faster for the unconstrained control case than for the constrained one. Constrained control is active up to time $t=4$ and after that it steadily tends to zero, see Figure \ref{fig:ex1.1}(ii). The  state cost $\norm{x(t)}^2+\norm{y(t)}^2$, documented in Figure \ref{fig:ex1.1}(iii), does not settle to zero in the uncontrolled case, whereas both state and control cost change their transient behavior to settle at zero both for the constrained and unconstrained HJI synthesis given in Figure \ref{fig:ex1.1}(iii)-(iv). The uppermost solid curve in Figure \ref{fig:ex1.1}(v) is related to the convergence of the state cost for the constrained control case with $\gamma=\gamma^*=7.9$. Further, if we increase $\gamma$ i.e. for $\gamma>\gamma^*$, the state cost for the constrained HJI converges to the constrained HJB  state cost which is also observed in Figure \ref{fig:ex1.1}(v). Similar transient behavior for the convergence of the HJI constrained control cost  to the HJB constrained control cost is shown in Figure \ref{fig:ex1.1}(vi). We also observed similar behavior for the unconstrained control case.
\begin{figure}[!h]
	\centering
	(i)	\includegraphics[width=0.45\textwidth]{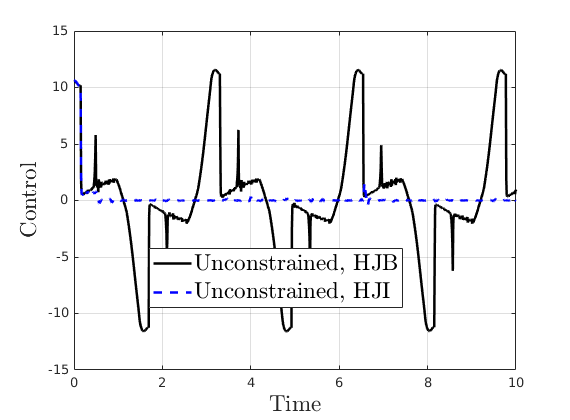}
	(ii)\includegraphics[width=0.45\textwidth]{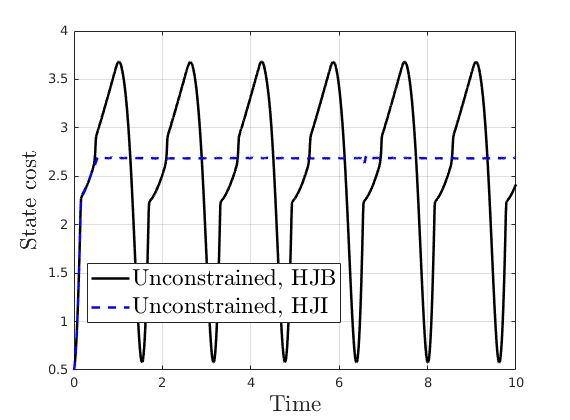}
	\caption{Test 4: Deterministic unconstrained case, comparison with the HJB case, $w_1=-0.6u_{HJB}+0.01sin(2t)$, $[x_0,y_0]=[-0.5,-0.5]$.  {\bf i)} Control,  {\bf ii)}  State cost.}\label{fig:ex1.2}
\end{figure}
\begin{figure}[!h]
	\centering
	(i)	\includegraphics[width=0.45\textwidth]{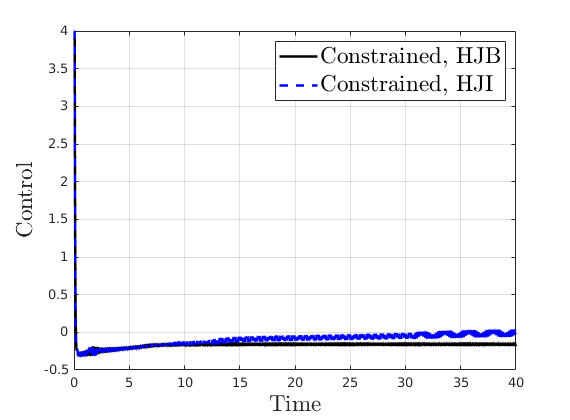}
	(ii)\includegraphics[width=0.45\textwidth]{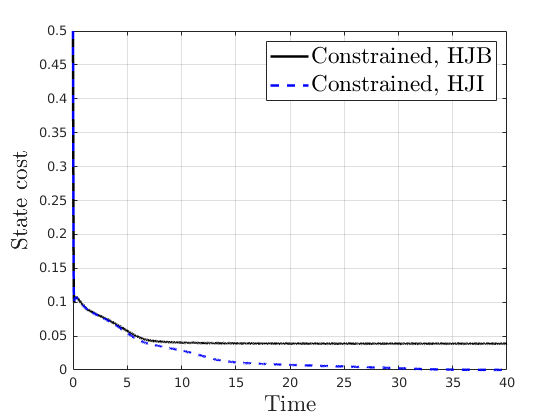}
	\caption{Test 4: Deterministic constrained case, comparison with the HJB case, $w_2=0.6u_{HJI}+0.1sin(40t)$, $[x_0,y_0]=[-0.5,-0.5]$,  $-4\leq u\leq 4$.  {\bf i)} Control,  {\bf ii)}  State cost.}\label{fig:exx1.2}
\end{figure}

Figures \ref{fig:ex1.2}-\ref{fig:exx1.2} show the temporal behavior of the  HJI and HJB synthesis against the worst-case scenario in presence of different initial conditions. With different types of additional noise, comparisons of the HJI-HJB synthesis are discussed. Figure \ref{fig:ex1.2}(i)-(ii) show that if the disturbance is of the form $w_1=-0.6u_{HJB}+0.01sin(2t)$, the control $u$ and the running state  cost $\norm{x}^2+\norm{y}^2$ oscillate for the unconstrained HJB case;  for the unconstrained HJI case they both are non-oscillatory. Concerning Figure \ref{fig:exx1.2}, where the disturbance is of the form $w_2=0.6u_{HJI}+0.1sin(40t)$ we set the control constraint as $-4\leq u\leq 4$, since $-1\leq u\leq 1$, for example, is too restrictive to observe the difference between the HJI and HJB case. We find that the constrained  HJI synthesis tends to zero whereas the constrained HJB fails, which is documented in Figures \ref{fig:exx1.2}(i)-(ii).

Concerning the  second-order HJI equation, we consider
$$\underset{u(\cdot)\in \cU}{\min}\; \underset{w(\cdot)\in \cW}{\max}\; \cJ(u(\cdot),w(\cdot);(x_0,y_0))=\E \int_{0}^{\infty}e^{-\lambda t}\Big(\norm{x}^2+\norm{y}^2+\norm{u}^2_{R}-\gamma^2\norm{w}^2_P\Big) dt,$$
subject to the following stochastic Van der Pol system
\begin{align}\label{eqx4.1}
	dx&=y \,dt,\quad x(0)=x_0,\notag\\
	dy&=\Big((1-x^2)y-x+u+w\Big)\,dt+0.1 \,dW, \quad -1\leq u\leq 1,\quad y(0)=y_0.
\end{align}
Here $R$, $P$, and $dt$ are the same as in the deterministic  case.\\
The dynamic behavior of engineering structures such as suspension bridges and offshore platforms, when subjected to random oscillations, is naturally described by second-order stochastic differential equations. Among these, variants of the Van der Pol oscillator play a crucial role, as their analysis is essential for developing robust control strategies in engineering and physical systems. These oscillators are central in the study of stability and bifurcations of stochastic dynamical systems. Within a game-theoretic framework, such systems can be interpreted as zero-sum stochastic differential games, where the controller seeks to stabilize the structure while nature or external disturbances act adversarially. 

\begin{figure}[!h]
	\centering
	(i)\includegraphics[width=0.45\textwidth]{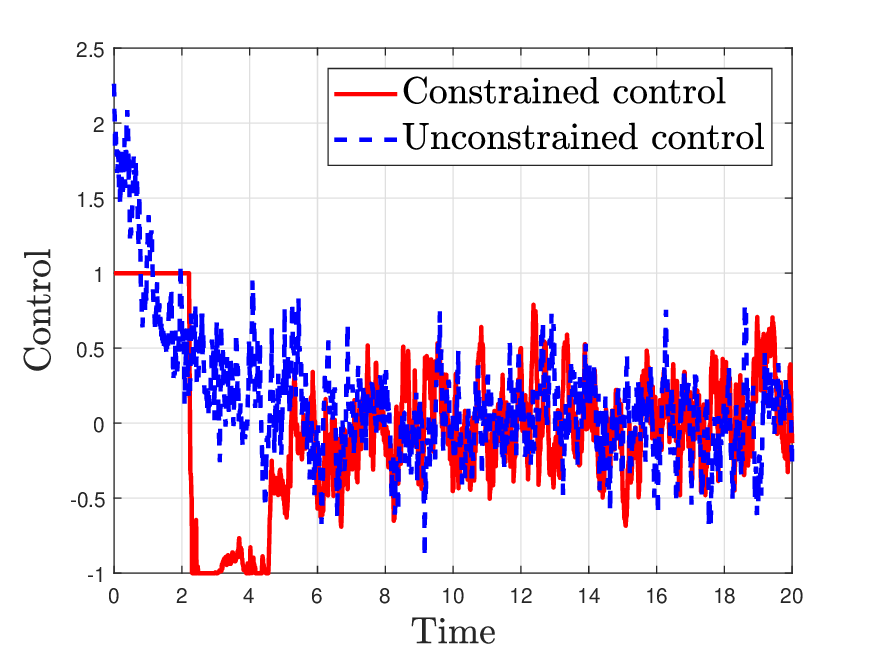}
	(ii)\includegraphics[width=0.45\textwidth]{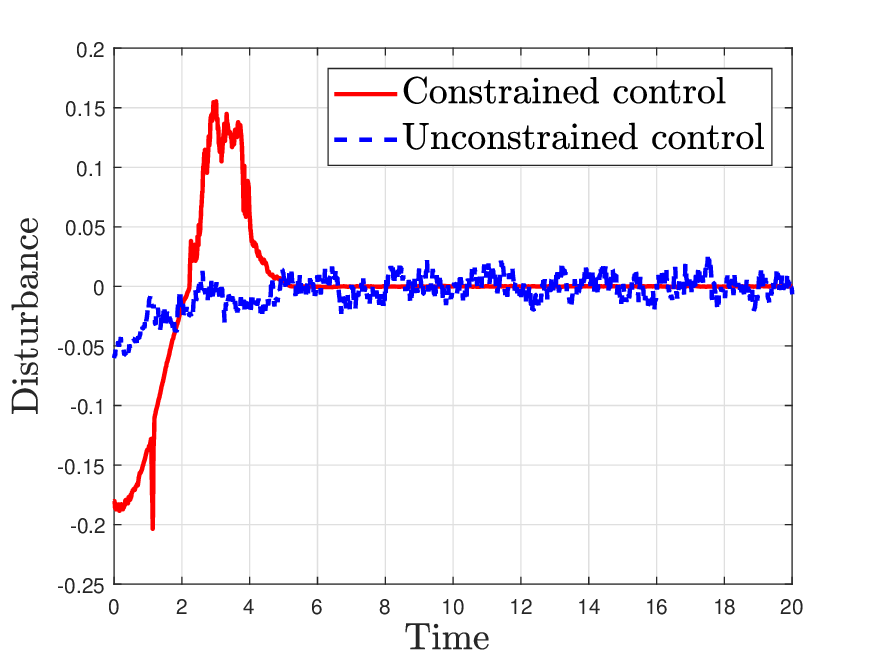}
	(iii)\includegraphics[width=0.45\textwidth]{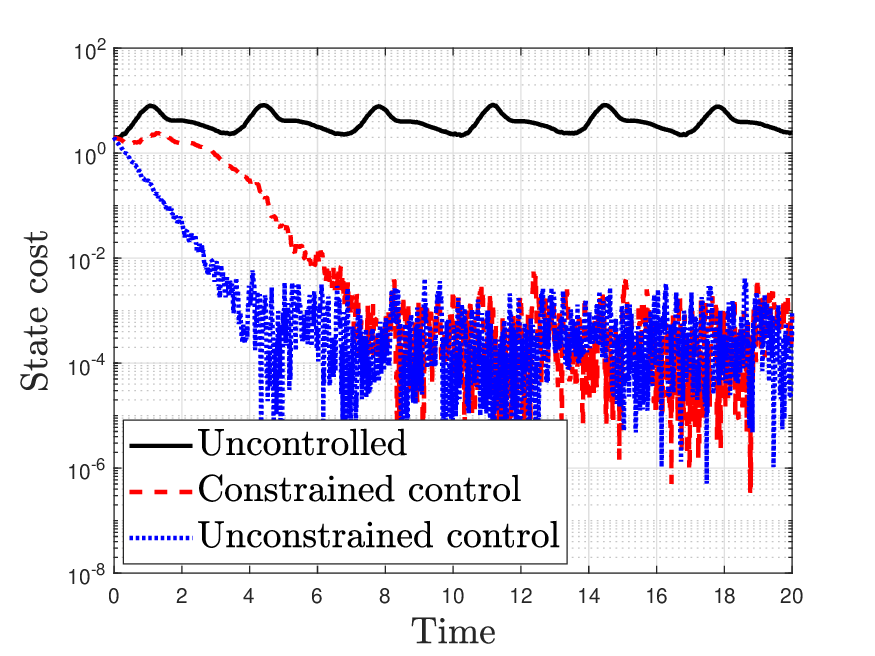}
	\caption{Test 4: Stochastic case, 2D Van der Pol equation.  {\bf i)} Control, {\bf ii)} Disturbance,  {\bf iii)} State cost. }\label{fig:ex1.3}
\end{figure}

Solving the second-order HJI equation, we get $\gamma^*=0.61$ for the unconstrained control case and $\gamma^*=4.89$ for the constrained case. The constrained and unconstrained  HJI controls under stochastic dynamics   can be   compared in Figure \ref{fig:ex1.3} (i). The corresponding state running cost
converges to zero in both cases, with a faster decay rate for the unconstrained HJI case.
The state cost for the uncontrolled and controlled cases can be compared  in Figure
\ref{fig:ex1.3} (iii). Figure \ref {fig:ex1.3} (ii) shows that the disturbance for both, the constrained and unconstrained HJI synthesis,  settle to zero.

	So far the discount factor was fixed at $\lambda = 0.005$. We next represent some results for the case when the solution for $\lambda = 0.005$ is used as initialization with $\lambda=0$. For these calculations  we choose  $dt=0.01$ . In the unconstrained case, with $\gamma=1$ we find  $\norm{V_0-V_{0.005}}_{L^\infty(\Omega)}=0.009$. Increasing $\lambda$ to  $\lambda=0.05$, we found that the error increases to  $\norm{V_0-V_{0.05}}_{L^\infty(\Omega)}=0.08$. We also computed the optimal $\gamma$ values,  and obtained  $\gamma^*=0.59$ and $\gamma^*=0.37$ for  $\lambda=0.005$ respectively $\lambda=0$, along  with the $L^\infty$ error $\norm{V_0-V_{0.005}}_{L^\infty(\Omega)}=0.07$. Similarly, in the constrained case we calculated $V_{0}$ and $V_{0.005}$ with $\gamma=10$, keeping the other parameters the same as in the unconstrained case. We obtained  $\norm{V_0-V_{0.005}}_{L^\infty(\Omega)}=0.089$. Thus, at the expense of decreasing $dt$, successful  results for $\lambda =0$ were obtained when initializing the algorithm with solutions from $\lambda>0$.


\section*{Concluding remarks}
Policy iteration for two-player zero-sum differential games  related to $\cH_{\infty}$ control problem isdiscussed for both deterministic and stochastic systems. Using an implicit upwind scheme for backward PDEs, numerical results for the first and second-order HJI cases are presented with a detailed comparison for the constrained and unconstrained  HJB synthesis. It would be of interest to extend the approach in  the case of two-player nonzero-sum games, where instead of the HJI equation, we have coupled HJB equations corresponding to the so-called Nash equilibrium. Our convergence analysis relies on regularity assumptions (Assumptions 1 and 2), consistent with standard practice in the literature, though not generally verifiable from simple problem data. The example in Section 4 shows that the value function may fail to be differentiable everywhere, especially near domain boundaries. Establishing rigorous conditions that guarantee $C^1$ or $C^2$ differentiability of the value function is a promising research direction. Since only viscosity solutions can be expected in full generality, relaxing the $C^1$/$C^2$ requirements and establishing convergence of policy iteration in that framework remains a challenging open problem, which we identify as another interesting direction for future work.
\section*{Acknowledgments}
The authors gratefully acknowledge support by the ERC advanced grant 668998
(OCLOC) under the EU's H2020 research program. Sudeep Kundu also gratefully acknowledges the support of the Science \& Engineering Research Board (SERB), Government of India, under the Start-up Research Grant, Project No. SRG/2022/000360.

%
\end{document}